\RequirePackage{etex}
\documentclass[11pt]{amsart} 

\usepackage[utf8]{inputenc}
\usepackage{hyperref}
\hypersetup{
	colorlinks   = true, 
	urlcolor     = blue, 
	linkcolor    = blue, 
	citecolor   = blue
}
\usepackage{lineno}
\usepackage{silence}
\usepackage{tikz-cd}
\usepackage{mathtools}
\usepackage{amsmath}
\usepackage{amsthm}
\usepackage{enumerate}
\usepackage{amsfonts}
\usepackage{amssymb}
\usepackage{array}
\usepackage[capitalise]{cleveref}
\usepackage[margin=2.7cm]{geometry} 
\usepackage{enumitem}

\usepackage{longtable}
\usepackage{graphicx}
\usepackage{xypic}
\usepackage{faktor}
\usepackage{colortbl}
\usepackage{tikz}
\usepackage{autonum}
\usepackage[hang,flushmargin]{footmisc} 
\usepackage[normalem]{ulem}
\usepackage{cancel}

\theoremstyle{plain}
\newtheorem{theorem}{Theorem}[section]
\newtheorem*{assumption}{Assumption}
\newtheorem{proposition}[theorem]{Proposition}
\newtheorem{question1}[theorem]{Question}
\newtheorem{corollary}[theorem]{Corollary}
\newtheorem{lemma}[theorem]{Lemma}

\newtheorem{question}{Question}

\theoremstyle{definition}
\newtheorem{definition}[theorem]{Definition}
\newtheorem{example}[theorem]{Example}
\newtheorem{remark}[theorem]{Remark}

\DeclareMathOperator{\GL}{GL}
\DeclareMathOperator{\PGL}{PGL}

\DeclareMathOperator{\Aut}{Aut}
\DeclareMathOperator{\Gal}{Gal}

\DeclareMathOperator{\diag}{diag}
\DeclareMathOperator{\End}{End}

\DeclareMathOperator{\Mat}{Mat}
\DeclareMathOperator{\Tr}{Tr}

\DeclareMathOperator{\rk}{rk}
\DeclareMathOperator{\Sym}{Sym}

\DeclareMathOperator{\Hom}{\mathrm{Hom}}
\DeclareMathOperator{\Proj}{\mathrm{Proj}}
\DeclareMathOperator{\Spec}{\mathrm{Spec}}

\DeclareMathOperator{\val}{val}
\DeclareMathOperator{\MG}{M}
\DeclareMathOperator{\Ann}{Ann}

\newcommand{\Q}{\mathbb{Q}}
\newcommand{\Z}{\mathbb{Z}}

\newcommand{\F}{\mathbb{F}}

\newcommand{\K}{K}
\newcommand{\LL}{L}

\newcommand{\MM}{\mathbb{M}}

\newcommand{\Ccal}{\mathcal{C}}

\newcommand{\Ocal}{\mathcal{O}}

\newcommand{\Lcal}{\mathcal{L}}
\newcommand{\Kcal}{\mathcal{K}}
\newcommand{\Mcal}{\mathcal{M}}
\newcommand{\mO}{\mathcal O}

\newcommand{\mfrak}{\mathfrak{m}}
\newcommand{\pfrak}{\mathfrak{p}}
\newcommand{\qfrak}{\mathfrak{q}}
\newcommand{\Bfrak}{\mathfrak{B}}
\newcommand{\m}{\mathfrak m}

\newcommand{\Norm}[2]{\mathrm{N}_{#1/#2}}

\newcommand{\Trl}{\Tr_{\LL/\K}}

\newcommand{\coker}{\mathrm{coker}}
\newcommand{\rki}{\mathrm{rk}_{\mathrm{inn}}}
\newcommand{\frk}{\mathrm{frk}}
\newcommand{\Res}{\mathrm{Res}}
\newcommand{\cor}[1]{\mathcal{#1}}
\newcommand{\gen}[1]{\langle #1\rangle}

\makeatletter
\newcommand{\xdashrightarrow}[2][]{\ext@arrow 0359\rightarrowfill@@{#1}{#2}}
\newcommand{\xdashleftarrow}[2][]{\ext@arrow 3095\leftarrowfill@@{#1}{#2}}
\newcommand{\xdashleftrightarrow}[2][]{\ext@arrow 3359\leftrightarrowfill@@{#1}{#2}}
\def\rightarrowfill@@{\arrowfill@@\relax\relbar\rightarrow}
\def\leftarrowfill@@{\arrowfill@@\leftarrow\relbar\relax}
\def\leftrightarrowfill@@{\arrowfill@@\leftarrow\relbar\rightarrow}
\def\arrowfill@@#1#2#3#4{%
	$\m@th\thickmuskip0mu\medmuskip\thickmuskip\thinmuskip\thickmuskip
	\relax#4#1
	\xleaders\hbox{$#4#2$}\hfill
	#3$
}

\@namedef{subjclassname@2020}{%
  \textup{2020} Mathematics Subject Classification}
\makeatother

\usepackage{etoolbox}
\makeatletter
\patchcmd{\@setaddresses}{\indent}{\noindent}{}{}
\patchcmd{\@setaddresses}{\indent}{\noindent}{}{}
\patchcmd{\@setaddresses}{\indent}{\noindent}{}{}
\patchcmd{\@setaddresses}{\indent}{\noindent}{}{}
\makeatother

\title{Valued rank-metric codes}

\author[Y.~El Maazouz]{Yassine El Maazouz}

\address{Y.~El Maazouz: 
RWTH Aachen, Chair of Algebra and Number Theory,
Pontdriesch 14/16, Office 110,
D-52056 Aachen,
Germany.}
\email{yassine.el-maazouz@rwth-aachen.de}

\author[M.~A.~Hahn]{Marvin Anas Hahn}
\address{M.~A.~Hahn: School of Mathematics 17, Westland Row, Trinity College Dublin, Dublin 2, Ireland.}
\email{hahnma@tcd.ie}

\author[A.~Neri]{Alessandro Neri}
\address{A.~Neri: Department of Mathematics: Analysis, Logic and Discrete Mathematics, Ghent University, Krijgslaan 281, 9000 Ghent, Belgium.}
\email{alessandro.neri@ugent.be}

\author[M.~Stanojkovski]{Mima Stanojkovski}
\address{M.~Stanojkovski:
Dipartimento di Matematica, Universit\`a di Trento, Via Sommarive  14, 38123 Povo di Trento (TN), Italy.
}
\email{mima.stanojkovski@unitn.it}

\date{}

\begin{document}

	\begin{abstract}
		In this paper, we study linear spaces of matrices defined over discretely valued fields and discuss their \emph{dimension} and\ \emph{minimal rank} drops over the associated residue fields. To this end, we take first steps into the theory of rank-metric codes over discrete valuation rings by means of skew algebras derived from Galois extensions of rings. Additionally, we model \emph{projectivizations} of rank-metric codes via Mustafin varieties, which we then employ to give sufficient conditions for a decrease in the dimension.
	\end{abstract}

	\maketitle

  {
  \tableofcontents
  }

\section{Introduction}

\noindent
Linear spaces of matrices are a classical topic in mathematics with a wide array of applications in various fields, such as linear algebra \cite{MR3672967,MR2353110, MR2154253}, optimization and semidefinite programming \cite{MR3075433, MR3015090}, algebraic geometry \cite{MR3107531, MR2504735},  and probability theory \cite{MR3186444}. A particularly interesting application of linear spaces of matrices stems from algebraic coding theory in the study of so-called \textit{rank-metric codes}. Originally introduced in \cite{MR514618}, they became very popular in the last decade, due to their application to random network coding \cite{MR2450762}. Because of their numerous applications, rank-metric codes have been studied  from various points of view, showing connections with many topics in finite geometry \cite{MR3620728}, algebra \cite{sheekey2020new,byrne2019tensor} and combinatorics \cite{gluesing2020partitions}; see also \cite{sheekey201913}. The key question in this line of research is the following:\vspace{3pt}
\begin{center}
\em Given a field $K$ and positive integers $m,n,d,r$, is there a linear space of matrices $\mathcal{L}\subseteq\mathrm{Mat}_{m\times n}(K)$ of dimension $d$, such that all $A\in\mathcal{L}$ satisfy $\mathrm{rank}(A)\ge r$? \vspace{3pt}
\end{center}
Such a linear space is called a $(d,r)$-\textit{rank-metric code in} $\mathrm{Mat}_{m\times n}(K)$.
This question has been well-studied when $K$ is a finite field \cite{MR514618, MR791529}, when $K$ admits a degree $n$ cyclic Galois extension \cite{MR1306980, MR2503966, augot2013rank}, over the reals \cite{MR139178,adams1965matrices,MR1378837} and over algebraically closed fields \cite{MR296081,MR1093747}.
In this work, we make first steps in the theory of rank-metric codes over local fields and study their relation to rank-metric codes over the finite rings associated to the ideals' filtration. 

\subsection{Algebraic set-up}
In the current section, we introduce the notation that will be valid throughout the paper.
Let $K$ be a \emph{discretely valued field}, i.e.\ the field $K$ is endowed with a map $\val_K :K \to \mathbb{Z} \cup\{+ \infty\}$ that, for all $a,b\in K$, satisfies the following:
	 \begin{align}
            &(1)\, \, \, \val_K(a) = + \infty \textrm{  if and only if  } a = 0 ,\\
            &(2)\, \, \, \val_K(ab) = \val_K(a)  + \val_K(b),\\
            &(3)\, \, \, \val_K(a+b)\ge\mathrm{min}( \val_K(a), \val_K(b)),\\
            &(4)\, \, \, \val_K(K)\not=\{0,+\infty\}.
	 \end{align}
	 The last condition simply excludes the trivial valuation on $K$. The \emph{valuation ring} $\mathcal{O}_K$ of $K$, also called the \emph{ring of integers} with respect to $\val_\K$, 
	 is defined by
	 \[
	 \mathcal{O}_K = \{x\in K\mid\val_K(x)\ge0\}
	 \]
	 and is a local principal ideal domain with unique maximal ideal $$\mathfrak{m}_K =\{x\in K\mid \val_K(x)>0\}.$$ Moreover, the maximal ideal $\mathfrak{m}_K$ is generated by any element $\pi \in K$ with $\val_K(\pi)=1$. Any such element is called a \textit{uniformizer}. For the rest of this paper, we fix once and for all a uniformizer $\pi$. We denote the \textit{residue field} by
	 $$\overline{K}=\faktor{\mathcal{O}_K}{\mathfrak{m}_K}$$ 
	 and use the bar notation for the objects in $\overline{\K}$. For each positive integer $k$, we write
	 \[
	 \Res(\K,k)=\mO_{\K}/\m_{\K}^k, 
	 \]
	 for the \emph{$k$-th residue ring} associated to $K$ 
	 and observe that  $\overline{\K}=\Res(\K,1)$. For readers not familiar with discretely valued fields, we collect below a couple of fundamental examples.

	 \begin{example}\
	 	\begin{enumerate}[label=$(\arabic*)$]
	 		\item Let $p$ be a prime number and let $K=\mathbb{Q}$ be the field of rational numbers. We endow $K$ with the \emph{$p$-adic valuation} $\val_\K=\val_p$ 
	 		 that associates to each $a/b\in\Q$, with $a,b \in \Z$, the integer
	 		 \[
	 		 \val_p(a/b)=\max\{ \alpha : p^\alpha \textup{ divides }a\}-\max\{\beta : p^{\beta}\textup{ divides }b \}.
	 		 \]
	 	 The ring of integers associated to $\val_p$ is 
	 	 \[
	 	 \mO_{\K}=\{ a/b\in \Q : a,b\in\Z \textup{ coprime, } p \textup{ does not divide } b \}
	 	 \]
	 	  and its unique maximal ideal is $\mathfrak{m}_{\K} = p \mO_{\K}$, generated by the uniformizer $\pi = p$. For each positive integer $k$, we have that $\Res(\K,k)$ is isomorphic to $\Z/p^k\Z$ and in particular $\overline{\K}=\F_p$ is the finite field of $p$ elements. 
	 		
	 		\item Let $K$ be the field of formal Laurent series in the variable $\pi$ with rational coefficients, that is
	 		\begin{equation}
	 			K=\mathbb{Q}((\pi))=\left\{\sum_{j\ge l}a_j \pi^j\mid a_j\in\mathbb{Q}, l\in\mathbb{Z} \text{ and } a_l \neq 0 \right\}.
	 		\end{equation}
	 		The valuation of an element $f = \sum_{j\ge l}a_j\pi^j \in K$ with  $ a_l \neq 0$ is $l$, in other words $\val_{\K}(f)$ is the lowest exponent of $\pi$ in the series (which we take to be  $+\infty$ when $f = 0$). Then, the ring of integers of $K$ is the ring of formal power series with rational coefficients, i.e.
	 		\begin{equation}
	 			\mathcal{O}_K=\mathbb{Q}[[\pi]]=\left\{\sum_{j\ge l}a_j \pi^j\mid a_j\in\mathbb{Q}, l\in\mathbb{Z}_{\ge0} \text{ and } a_l \neq 0 \right\},
	 		\end{equation}
	 		and its maximal ideal
	 		\begin{equation}
	 			\mathfrak{m}_K = \pi \Q[[\pi]] = \left\{\sum_{j\ge l}a_j\pi^j\mid a_j\in\mathbb{Q}, l\in\mathbb{Z}_{>0} \text{ and } a_l \neq 0\right\}
	 		\end{equation}
	 		is generated by the uniformizer $\pi$. The residue field $\overline{K}$ is isomorphic to $\mathbb{Q}$.
	 	\end{enumerate}
	 \end{example} 
	 
	 \noindent
	 In the examples we just considered, the residue field $\overline{\K}$ happened to be perfect. This is not always the case. For the sake of brevity, in this paper we will assume that $\K$ is a \emph{local field}, i.e.\ that $\overline{\K}$ is finite and $K$ is complete with respect to its valuation.\footnote{This assumption is only needed for the existence of the Teichm\"uller representatives in Section \ref{sec:rank-drops}. Otherwise, the results hold for any discretely valued field.}
	 
	 Let $\LL/\K$ be an unramified Galois extension of (finite) degree $n$ and Galois group $$G=\Gal(\LL/\K);$$
	 then $G$ is cyclic and we fix a generator $\sigma$ of $G$. 
	 Let $\val_\LL$ be the discrete valuation on $\LL$ that uniquely extends $\val_\K$, that is, for each $x \in \LL$, one has
	 $$\val_\LL(x)=\frac{1}{n}\val_\K(\Norm{\LL}{\K}(x)).$$
	 Let $\mO_\LL$ be the ring of integers with respect to $\val_\LL$ and let $\m_\LL$ denote its unique maximal ideal. 
	 Denote by $\overline{\LL}$ the residue field $\mO_\LL/\m_\LL$ and note that, $\LL/\K$ being unramified, $\m_\LL=\mO_{\LL} \pi$. For each positive integer $k$, let $\Res(\LL,k)$ 
	 be the analogue of $\Res(\K,k)$ for $\LL$. Since the extension $\LL/\K$ is unramified, its Galois group $G$ is isomorphic to $\Gal(\overline{\LL}/\overline{\K})$ and we identify these two groups.
	 
	 The global setting is summarized in the following diagram:
	   
	   \[
	   \xymatrix{
	   \overline{\LL} \ar@{-}[dd]_{G} & & \m_{\LL}\ar[rr] \ar@{-}[dd] & & \mO_{\LL}\ar@{-}[dd]\ar[rr] & & \LL\ar@{-}[dd]^{G} \\ \\
	   \overline{K} & & \m_{\K}\ar[rr] && \mO_{\K}\ar[rr] && \K
	   }
	   \]
	 
\subsection{First steps in valued rank-metric codes}

Our main focus in this work is the following question.
\begin{question}\label{quest:rank_metric_residue}
Let $\mathcal{L}\subseteq\mathrm{Mat}_{m\times n}(K)$ be a $(d,r)$-rank-metric code over a discretely valued field $K$ and let 
    \[
        \mathcal{K}=\mathcal{L}\cap\mathrm{Mat}_{m\times n}(\Ocal_K).
    \]  
Is the projection $\overline{\mathcal{K}}$ of $\mathcal{K}$ modulo
 $\pi$ a $(d,r)$-rank-metric code?
\end{question}

\noindent
 In general the answer is no. We take first steps towards a conceptual framework to study which linear spaces $\mathcal{L}$ remain $(d,r)$-rank-metric codes when passing  from $K$ to its residue field $\overline{K}$. To this end, we need to control both the dimension and the minimal rank drop from $\Lcal$ to $\overline{\Kcal}$. We outline two approaches to this problem, the first one focusing on the behaviour of the minimal rank in the linear space when passing to $\overline{K}$ and the second on the behaviour of the dimension.

\subsection{Field extensions and rank-metric codes}
The first approach we propose to deal with \cref{quest:rank_metric_residue} relies on a skew algebra framework. More precisely, whenever  a ring $R$ admits a (cyclic) $G$-Galois extension $S$ of degree $n$, the skew algebra $S[G]$ is isomorphic to $\End_R(S)$, which is in turn isomorphic to $\Mat_{n}(R)$.
We remark that, in the context of \cref{quest:rank_metric_residue}, working with square matrices is not restrictive: indeed restricting $R$-endomorphisms of $S$ to a chosen free $R$-submodule of fixed rank $m$ yields $m\times n$ (rectangular) matrices, 
obtaining $(d,r)$-rank-metric codes in $\Mat_{m\times n}(R)$ from $(d,r+n-m)$-rank-metric codes in $\Mat_{n}(R)$.
This isomorphic representation of the matrix space has been used for constructing $(d,r)$-rank-metric codes with high minimum rank $r$ over finite fields in \cite{MR514618, MR791529}, general fields in \cite{MR1306980, MR2503966, augot2013rank, ACLN20+} and finite principal ideal rings in \cite{MR4038895}. Every square matrix can be represented as a \emph{$\sigma$-polynomial}, that is an $S$-linear combination of the powers of $\sigma$.  The advantages of this viewpoint are multiple. First, the $\sigma$-polynomial representation gives a natural lower bound on the rank of an endomorphism. Furthermore, $S[G]$ is isomorphic to a quotient of the skew polynomial ring $S[x;\sigma]$, where operations are computationally efficient; see e.g. \cite{MR1503119,MR2006329}. This is a crucial point exploited  for the development of fast decoding algorithms. Finally, this approach provides a natural notion of $S$-linearity of some spaces $\mathcal L \subseteq \Mat_{m\times n}(R)$. When $S/R$ is an extension of finite fields, $S$-linear codes have been shown to have generically the best parameters \cite{MR3756146,byrne2020partition}, while this is not the case for the $R$-linear ones \cite{antrobus2019maximal,byrne2020partition,gluesing2020sparseness,gruica2020common}.

In Section \ref{sec:skewalgebras} we focus on $G$-Galois extensions of local principal ideal rings. We first show that the skew algebra representations are naturally compatible with the canonical filtration given by the ideals in $\mO_\K$; cf.\ \cref{thm:compatibility}. 
We then prove that, also within this context, the $\sigma$-polynomial representation gives a natural lower bound on the rank of the corresponding endomorphism; cf.\ \cref{prop:innerrank_lowerbound}. As a consequence, we define the counterpart of Gabidulin and twisted Gabidulin codes over local principal ideal rings and show that, also in this more general framework, they are maximum rank distance (MRD) codes; cf.\ \cref{thm:sheekey_rings}.
Furthermore, we show in \cref{thm:MRDiff} that the parameters of a free rank-metric code over a finite local principal ideal ring are completely determined by its image over the residue field. We investigate the case of infinite rings in the last section. 

\subsection{Dimensions via convex hulls in Bruhat-Tits buildings}
In \cref{sec-bruhat}, we  establish a connection between the theory of rank-metric codes and so-called \textit{Bruhat-Tits buildings}. Bruhat-Tits buildings arise from the geometric representation theory of reductive groups. In this work, we focus on the Bruhat-Tits building $\mathfrak{B}_d$ associated to $\PGL(V)$, where $V$ is a $d$-dimensional vector space over $K$. The Bruhat-Tits building $\mathfrak{B}_d$ is equipped with a notion of convexity. In \cref{prop-must}, we derive a criterion in terms of this convexity for a decrease in dimension of a given linear space of matrices when passing to the residue field.

Our main tool is the notion of \textit{Mustafin varieties}, which was first introduced by Mustafin in \cite{mustafin1978nonarchimedean} in order to generalise Mumford's seminal work on the uniformization of curves to higher dimension. Mustafin varieties have ever since developed a life on their own, with applications in the theory of Shimura varieties \cite{MR1827019,gora2019local}, Chow quotients of Grassmannians \cite{keel2006geometry}, tropical geometry \cite{MR2861606, hahn2020mustafin}, Brill-Noether theory \cite{he2019degenerations}, $p$-adic nonabelian Hodge theory \cite{hahn2019strongly,hahn2020mustafinmodels} and computer vision \cite{hahn2020one}. Mustafin varieties are degenerations of projective spaces over the discretely valued field $K$ whose limits are varieties over the residue field $\overline{K}$. A key feature of Mustafin varieties is that their limits may be studied in terms of convex hulls in $\mathfrak{B}_d$. The key observation for our proof of \cref{prop-must} is that Mustafin varieties in some sense model the \textit{multi-projectivization} of rank-metric codes. The multi-projectivization of linear spaces, i.e.\ the closure of an affine linear space in a product of projective spaces, is a well-studied topic algebraic geometry \cite{zbMATH06547829}. A careful examination of the classification of Mustafin varieties obtained in \cite{hahn2020mustafin} reveals that, generically, linear spaces of matrices whose dimensions remain unchanged when passing to the residue field correspond to a unique distinguished component of a special Mustafin variety; cf.\ \cref{prop-metricmust}. This irreducible component corresponds to a unique vertex of a certain convex hull in $\mathfrak{B}_d$. In \cref{prop-must}, we distinguish this vertex in a special case which yields a certificate for a drop in the dimension when passing to the residue field.

\subsection{Structure of the paper}
In \cref{sec:skewalgebras}, we develop our conceptual approach via skew algebras. We review basic notions and establish compatibility results for skew algebras in terms of the canonical filtration given by the powers of the maximal ideal; cf.\ \cref{thm:compatibility}. Furthermore, we initiate the study of rank-metric codes over local principle ideal domains and generalise various classical constructions to this new setting; cf.\ \cref{thm:sheekey_rings}. In \cref{sec-bruhat}, we study MRDs in relation to Bruhat-Tits buildings and Mustafin varieties. This allows us to develop a sufficient criterion for a drop in dimension of a given MRD when passing to the residue field; cf. \cref{prop-must}.

\addtocontents{toc}{\protect\setcounter{tocdepth}{1}}
\subsection*{Acknowledgements}
    \noindent
     The authors wish to thank Bernd Sturmfels and the organizers of the \emph{LSSM working group} at the Max-Planck-Institute for Mathematics in the Sciences (MPI MiS), which has inspired this project in the initial form of ``Question 58'' \cite{LSSM}. The first author is partially supported by the EWJ Gateway Fellowship at U.C.\ Berkeley. The third author is supported by  the Swiss National Science Foundation through grant no.\ 187711. The authors also wish to thank the anonymous referee for their comments, which led to an improvement of the paper.
\addtocontents{toc}{\protect\setcounter{tocdepth}{2}}

	 \section{The framework of skew algebras}
	 \label{sec:skewalgebras}
	 
	 \noindent
	 In this section, we study linear spaces of matrices from the point of view of skew algebras. We show a way of expressing elements in $\Mat_n(\K)$ as \emph{$\sigma$-polynomials} where $\sigma$ is the generator of a cyclic Galois extension $\LL$ of $\K$ of degree $n$. One of the properties of the last construction is that, if $A\in\Mat_n(\K)$ is associated to the $\sigma$-polynomial
	 \[
	 f_A=a_0+a_1\sigma + \ldots + a_k\sigma^k \in \LL[G]
	 \]
	 of degree $k\leq n-1$, then the rank of $A$ is known to be at least $n-k$. We will show how one can compute $f_A\in S[G]$ for any $A$ given in $\Mat_n(R)$, where $S/R$ is a \emph{Galois extension of rings}, and how these constructions are \emph{compatible in terms of reduction}. We will use this last fact to discuss rank drops and consider some famous examples in the context of \emph{MRD codes}.
	 
	 \noindent
	 The present section is organized as follows. In Section \ref{sec:dual}, we present some basic results in connection to \emph{integral bases}  of $\LL/\K$. Section \ref{sec:0top} is devoted to Galois extensions of rings and associated skew algebras. \cref{thm:compatibility} is the main result of this section and relates skew algebras associated to $\mO_{\LL}/\mO_{\K}$ and  $\Res(\LL,k)/\Res(\K,k)$, respectively. In Section \ref{sec:rank-drops}, we give a way of associating $\pi$-adic expansions to elements $f$ in $\Mat_n(\K)$ and discuss rank-drops when $f$ is viewed modulo a power of the maximal ideal $\m_{\K}$. In Section  \ref{sec:rank_polynomials} we generalize the theory of rank-metric codes to local principle ideal rings: the main result of this section is Theorem \ref{th:norm}, which 
	 allows us to generalize finite field constructions of MRD codes over local principal ideal rings in Section \ref{sec:gabidulin}. In the same section, we prove, moreover, that when the ring is finite, then the image of a free code over the associated residue field determines whether the code is MRD or not. 
	 
	 \subsection{Dual and integral bases}\label{sec:dual}
	 
	 Let $\Trl$ denote the \emph{trace} of $\LL /\K$, i.e.\ the map
	 $$%\begin{array}{rcl}
	 	\Trl: \LL  \longrightarrow  \K, \quad
	 	\alpha  \longmapsto  \sum_{g\in G} g(\alpha).
	 %\end{array}
	 $$
	 Since the extension $\LL / \K$ is separable, the $\K$-bilinear form 
	 $$\LL \times \LL \longrightarrow \K, \quad (\alpha,\beta)\longmapsto \Trl(\alpha\beta),$$
	that is induced by the trace map is nondegenerate.  As a consequence, any $\K$-basis $\alpha=(\alpha_1,\ldots,\alpha_n)$ of $\LL$ admits a unique \emph{dual basis}, that is a $\K$-basis $\alpha^*=(\alpha_1^*,\ldots,\alpha_n^*)$ of $\LL$ such that, for any $i,j\in\{1,\ldots,n\}$, one has $\Trl(\alpha_i\alpha_j^*)=\delta_{ij}$.
	 
	 \noindent
	 We fix an ordering on $G$ writing $G=\{g_1,\ldots,g_n\}$ and, for each vector $v=(v_1,\ldots,v_n)$ in $\LL^n$, we define the \emph{$G$-Moore matrix} $\MG_G(v)\in \Mat_n(\LL)$ as the matrix whose $(i,j)$-th entry is $g_i(v_j)$.
	 
	 \begin{lemma}\label{lem:GMoore}
	 	Let $\alpha=(\alpha_1,\ldots,\alpha_n)\in \LL^n$. Then the following hold:
	 	\begin{enumerate}[label=$(\arabic*)$]
	 		\item $\rk(\MG_G(\alpha))=\dim_{\K}\langle \alpha_1,\ldots,\alpha_n\rangle_\K$.
	 		\item if $\alpha$ is a $\K$-basis of $\LL$, then $\MG_G(\alpha)$ is invertible and we have:
	 		\begin{enumerate}[label=$(\alph*)$]
	 			\item $\MG_G(\alpha)^{-1}=\MG_G(\alpha^*)^\top.$
	 			\item for each $i\in\{1,\ldots, n\}$, one has $\val_{\LL}(\alpha_i)+\val_{\LL}(\alpha_i^*)\leq 0$.
	 		\end{enumerate}
	 	\end{enumerate}
	 \end{lemma}
	 
	 \begin{proof}
	 	This is \cite{ACLN20+} combined, for each $i\in\{1,\ldots, n\}$ with the following observation:
	 	\[0=\val_{\LL}(1)=\val_{\LL}\bigg(\sum_{g\in G}g(\alpha_i\alpha_i^*)\bigg)\geq \min_{g \in G}\{\val_\LL(g(\alpha_i\alpha_i^*))\}=\val_\LL(\alpha_i)+\val_{\LL}(\alpha_i^*).\]
	 \end{proof}

	 \begin{proposition}\label{prop:equivalence_basis}
	 	Let $\alpha=(\alpha_1,\ldots, \alpha_n)$ be a $\K$-basis of $\LL$ such that, for each $i\in\{1,\ldots,n\}$, one has  $\val_\LL(\alpha_i)\geq 0$. Then the following are equivalent:
	 	\begin{enumerate}[label=$(\arabic*)$]
	 		\item the image $\overline\alpha=(\overline\alpha_1,\ldots,\overline\alpha_n)$ of $\alpha$ in $\overline{\LL}$ is a $\overline\K$-basis of $\overline{\LL}$. 
	 		\item $\val_{\LL}(\det(\MG_G(\alpha)))=0$.
	 		\item $\MG_G(\alpha)\in \GL(n,\mO_\LL)$.
	 		%\item the columns/rows of $\MG_G(\alpha)$ are orthonormal with respect to $\| \cdot \|_\LL$.
	 		\item $\val_\LL(\alpha_i^*)=0$ for every $i$.
	 		\item {$\alpha$ is an $\mO_\K$-basis of $\mO_\LL$.}
	 	\end{enumerate}
	 	Moreover, if $\alpha$ satisfies $(1)-(5)$, then, for each $i\in\{1,\ldots,n\}$, one has  $\val_\LL(\alpha_i)=0$.
	 \end{proposition}
	 
	 \begin{proof}
	 	We start by showing that $(1) \Leftrightarrow (2)$. For this, let $\Delta=\det(\MG_G(\alpha))$ and note that having $\val_\LL(\alpha_i)\geq 0$ for each $i$ implies that $\val_\LL(\Delta)\geq 0$. Now, Lemma \ref{lem:GMoore} yields that $\overline\alpha$ is a $\overline\K$-basis of $\overline{\LL}$ if and only if $\det(\MG_G(\overline\alpha))\neq 0$.
	 	We conclude by observing that $\det(\MG_G(\overline\alpha))$ equals the image $\overline\Delta$ of $\Delta$ in $\overline{\LL}$ and so $\overline\alpha$ is a $\overline\K$-basis of $\overline{\LL}$ if and only if $\val_{\LL}(\Delta)=0$.
	 	
	 	The equivalence $(2) \Leftrightarrow (3)$ follows from basics of linear algebra. 
	 	To prove the equivalence of $(3)$ and $(4)$, we rely on Lemma \ref{lem:GMoore}(2.a)-(2.b): indeed, since $\MG_G(\alpha)^{-1}=\MG_G(\alpha^*)^\top$, we have that each element  $\alpha_i^*$ belongs to $\mO_\LL$ if and only if $\MG_G(\alpha)^{-1}\in \Mat_n(\mO_\LL)$.

	 	In order to show that $(3) \Rightarrow (5)$,  we let $\beta \in \mO_{\LL}$. Since $\alpha$ is a $\K$-basis of $\LL$, there exist unique $\lambda_1,\ldots,\lambda_n \in \K$ such that $\beta=\sum_i\lambda_i\alpha_i$. Fix such $\lambda_i$s and set $\lambda=(\lambda_1,\ldots,\lambda_n)$. Then we have $\lambda^\top=\MG_G(\alpha)^{-1}(g_1(\beta),\ldots,g_n(\beta))^\top$ and,  $\MG_G(\alpha)$ being invertible in $\Mat_{n}(\mO_{\LL})$, we derive that $\lambda \in \mO_{\K}^n$.
	 	
	 	We conclude the proof by observing that $(5)$ clearly implies $(1)$.
	 \end{proof}
	 
	 \begin{definition}
	 	A $\K$-basis of $\LL$ satisfying any of the equivalent properties in Proposition \ref{prop:equivalence_basis} is called an \emph{integral basis}.
	 \end{definition}
	 
	 Note that the existence of integral $K$-bases is ensured, for instance, by the fact that $\Ocal_L$ is a free module over $\Ocal_K$; cf.\ \cref{prop:equivalence_basis}(5).
	 
	 \subsection{Galois extensions of rings}\label{sec:0top}
	 Let $R$ be a commutative ring and let $S$ be a commutative $R$-algebra. Denote by $\End_R(S)$ the endomorphisms of $S$ as an $R$-module and by $\Aut(S/R)$ the collection of automorphisms of $S$ as an $R$-algebra. If $\phi:G\rightarrow\Aut(S/R)$
	 is a homomorphism, then the $\phi$-\emph{skew group algebra} of $G$ with coefficients in $S$ is the free $S$-module $S[G]_{\phi}$ with basis $G$ endowed with the following multiplication:
	 \begin{center}
	 	for $g,h\in G$ and $a_g,a_h\in S$, one has $(a_gg)(a_hh)=a_g \phi(g)(a_h) gh$.    
	 \end{center}
	 With this definition, it is not difficult to show $S[G]_\phi$ is a (possibly non-commutative) ring.
	 In our applications, $\phi$ will always be an isomorphism, which we will treat as the identity map, as in \cite{ACLN20+}, and we will therefore only write $S[G]$ instead of $S[G]_{\phi}$.
	 
	 \begin{example}\label{ex:skew-algebras}
	 	~
	 	\begin{enumerate}[label=$(\arabic*)$]
	 		\item When $\phi$ is the trivial map, then $S[G]_{\phi}$ is the usual group ring of $G$ over $S$. 
	 		\item When $R=\mO_\K$ and $S=\mO_\LL$, then the map
	 		\[\phi:G=\Aut(\LL/\K)\longrightarrow\Aut(\mO_{\LL}/\mO_{\K}), \quad f\longmapsto \phi(f)=f_{|\mO_\LL}\]
	 		is a well-defined group isomorphism.
	 		Indeed, thanks to \cite[Lem.~3.1]{MR1915966}, the elements of $\mO_{\LL}$ are integral over $\mO_{\K}$ and thus each element of $\Aut_{\K}(\LL)$ stabilizes $\mO_{\LL}$.
	 		Assuming $\phi$ to be the identity map, we get 
	 		\[
	 		\mO_{\LL}[G]=\mO_{\LL}[G]_{\phi}=\left\{\sum_{g\in G}a_g g \in \LL[G] : a_g \in \mO_{\LL}\right\}.
	 		\]
	 		\item For each nonnegative integer $k$, we give two equivalent ways of constructing what in the remaining part of this section will be denoted by $\m_{\LL}^k[G]$. We observe that the map 
	 		\[
	 		\tilde{\phi}: \Aut(\mO_{\LL}/\mO_{\K})\longrightarrow \Aut(\m_{\LL}^k/\mO_{\K}), \quad f \longmapsto \tilde{\phi}(f)=f_{|\m_\LL^k}
	 		\]
	 		is a well defined isomorphism. Set now $\phi:G\rightarrow \Aut(\m_{\LL}^k/\mO_{\K})$ to be $\tilde{\phi}$ precomposed with the isomorphism from (2). Then it follows that 
	 		\[
	 		\m_{\LL}^k[G]=\m_{\LL}^k[G]_{\phi}=\left\{\sum_{g\in G}a_gg \in \LL[G] : a_g\in\m_{\LL}^k\right\}.
	 		\]
	 	\end{enumerate}
	 \end{example}
	 
	 \begin{definition}
	 	A \emph{$G$-Galois extension} is an extension $S/R$ of commutative rings such that $G$ is a subgroup of $\Aut(S/R)$ and the following hold: 
	 	\begin{enumerate}[label=$(\arabic*)$]
	 		\item $S^G=\{s\in S : \textup{ for all $g\in G$ one has }g(s)=s\}=R$,
	 		\item the map $\tau: S\otimes_RS\rightarrow \oplus_{g\in G}S$, defined by
	 		\[x\otimes y \longmapsto \tau(x\otimes y)= (xg(y))_{g\in G}\]
	 		is bijective. 
	 	\end{enumerate}
	 \end{definition}
	 
	 \noindent
	 The following is a weaker version of \cite[Thm.~1.3]{ChHaRo/65}, which can also be found in \cite[Thm.~1.6]{Greither/92}.
	 
	 \begin{proposition}\label{prop:iso-phi}
	 	Let $S/R$ be an extension of commutative rings. Assume that $G$ is a subgroup of $\Aut(S/R)$ and that $R=S^G$. 
	 	Then the following statements are equivalent.
	 	\begin{enumerate}[label=$(\arabic*)$]
	 		\item The extension $S/R$ is $G$-Galois.
	 		\item The ring $S$ is a finitely generated projective $R$-module and the map 
	 		\[
	 		\varphi: S[G]\longrightarrow \End_{R}(S), \quad \mathbf{a}=\sum_{g\in G}a_gg \longmapsto (\varphi(\mathbf{a}): x \mapsto \sum_{g\in G}a_gg(x))
	 		\]
	 		is an isomorphism of $R$-algebras.
	 		\item For any nontrivial $g\in G$ and maximal ideal $\m$ of $S$, there exists $s\in S$ such that $g(s)-s$ is not in $\m$.
	 	\end{enumerate}
	 \end{proposition}

	 \begin{example}\label{ex:G-Galois}
	 	~
	 	\begin{enumerate}[label=$(\arabic*)$]
	 		\item The extension $\mO_{\LL}/\mO_{\K}$ is $G$-Galois. Indeed $G$ can be viewed as a subgroup of $\Aut(\mO_{\LL}/\mO_{\K})$ via the map $\phi$ from Example \ref{ex:skew-algebras}(2) and  $\mO_{\LL}^G=\mO_{\K}$ because $G$ is the Galois group of $\LL/\K$. The extension $\LL/\K$ being unramified, 
	 		Proposition \ref{prop:iso-phi}(3)  yields that $\mO_{\LL}/\mO_{\K}$ is $G$-Galois.
	 		\item For each nonnegative integer $k$, the extension $\Res(\LL,k)/\Res(\K,k)$ is $G$-Galois. The last inclusion is here to be interpreted in terms of the following natural commutative diagram
	 		\[
	 		\xymatrix{
	 			\mO_{\K}\ar[d] \ar[rd] \ar[r] & \ar[d] \mO_{\LL} \\
	 			\mO_{\K}/\m_{\K}^k \ar@{.>}[r] & \mO_{\LL}/\m_{\LL}^k,
	 		}
	 		\]
	 		where the dotted map is injective because $\LL/\K$ is unramified. For each $f$ in $ \Aut(\mO_{\LL}/\mO_{\K})$ denote by $\overline{f}_k$ the automorphism induced on $\Res(\LL,k)/\Res(\K,k)$ by $f$. Then the map \[
	 		\Aut(\mO_{\LL}/\mO_{\K}) \longrightarrow \Aut(\Res(\LL,k)/\Res(\K,k)), \quad f \longmapsto \overline{f}_k
	 		\]
	 		precomposed with the map $\phi$ from Example \ref{ex:skew-algebras}(2) gives an isomorphism $$G\longrightarrow\Aut(\Res(\LL,k)/\Res(\K,k))$$ and, moreover, we have $\Res(\LL,k)^G=\Res(\K,k)$. Applying Proposition \ref{prop:iso-phi}(3), we get that $\Res(\LL,k)/\Res(\K,k)$ is $G$-Galois.
	 	\end{enumerate}
	 \end{example}

	 \noindent
	 For each nonnegative integer $k$, define
	 $$\End_\K^k(\LL)=\{f\in\End_\K(\LL) \mid f(\mO_\LL)\subseteq \m_{\LL}^k\}.$$
	 
	 \begin{proposition}\label{prop:corresp-phi}
	 	For each nonnegative integer $k$, one has $$\varphi(\m_\LL^k[G])=\End_{\K}^k(\LL).$$
	 \end{proposition}
	 
	 \begin{proof}
	 	Fix $k$. %We claim that $\varphi(\m_\LL^k[G])=\End_{\K}^k(\LL)$. 
	 	The inclusion $\varphi(\m_\LL^k[G])\subseteq \End_{\K}^k(\LL)$ is clear so we show the other one. 
	 	First, choose an ordering on the elements of $G=\{g_1,\ldots, g_n\}$.
	 	Let $f \in \End_\K^k(\LL)$ and let $\mathbf{a}=\sum_i a_ig_i \in \LL[G]$  be such that $\varphi(\mathbf{a})=f$. Fix an integral $\K$-basis $\alpha=(\alpha_1,\ldots, \alpha_n)$ of $\LL$ and,
	 	for each index $j$, set  $$\beta_j=f(\alpha_j)=\varphi(\mathbf{a})(\alpha_j)=\sum_{i=1}^n a_i g_i(\alpha_j).$$
	 	This gives a linear system 
	 	$$(a_1,\ldots, a_n) \MG_G(\alpha)=(\beta_1,\ldots, \beta_n),$$
	 	which, thanks to Lemma \ref{lem:GMoore}, can be rewritten as
	 	\begin{equation}\label{eq:aiMoore}
	 		(a_1,\ldots, a_n)=(\beta_1,\ldots, \beta_n)\MG_G(\alpha^*)^\top.\end{equation}
	 	Now, for each $i\in\{1,\ldots,n\}$, we have $\val_{\LL}(\alpha_i^*)=0$ and so \eqref{eq:aiMoore} together with the definition of valuation yields $\val_{\LL}(a_j)\ge \min_i\{\val_{\LL}(\beta_i)\}\ge k$. This concludes the proof.
	 \end{proof}
	 
	 \noindent
	 Let $\varphi:\LL[G]\rightarrow \End_\K(\LL)$ denote the isomorphism given by Proposition \ref{prop:iso-phi}, where $\phi$ is taken to be the identity map.
	 For each nonnegative integer $k$, for $f\in\End_{\K}^0(\LL)$, and for $\mathbb{M}\in\{\K,\LL\}$, we write:
	 \begin{enumerate}[label=$(\arabic*)$]
	 	\item $\varphi_k$ instead of $\varphi_{|\m^k[G]}$,
	 	\item $\overline{\varphi}_k$ for the isomorphism $\Res(\K,k)[G]\rightarrow\End_{\Res(\LL,k)}(\Res(\LL,k))$ induced by $\varphi$,
	 	\item $\pi_{\MM}^{(k)}$ for the canonical projection $\mO_{\MM}\rightarrow \Res(\MM,k)$,  \item $\overline{f}_k$ for the induced endomorphism of $\Res(\LL,k)$.
	 \end{enumerate}
	 Note that the map in (2) is well-defined: indeed, thanks to Example \ref{ex:G-Galois}(2), the extension $\Res(\LL,k)/\Res(\K,k)$ is $G$-Galois and so $\overline{\varphi}_k$ exists by Proposition \ref{prop:iso-phi}. 
	 
	 If $V$ is an $n$-dimensional vector space over a field $\MM$ and  $\alpha$ is an $\MM$-basis of $V$, we denote by $\psi_{\alpha}$ the map  $\End_\MM(V) \rightarrow \Mat_n(\MM)$ that sends each endomorphism to its matrix representation with respect to the basis $\alpha$. 
	 
	 \begin{theorem}\label{thm:compatibility}
	 	Let $k$ be a nonnegative integer, and let $\alpha$ be an integral $\K$-basis of $\LL$. The following hold:
	 	\begin{enumerate}[label=$(\arabic*)$]
	 		\item $\m_\LL^k[G]$ is the kernel of the map 
	 		$$\mO_\LL[G]\longrightarrow \Res(\LL,k)[G], \quad \sum_{g\in G}a_gg\longmapsto \sum_{g\in G}\pi_{\LL}^{(k)}(a_g)g,$$ 
	 		\item $\End_{\K}^k(\LL)$ is the kernel of the map 
	 		$$\End_\K^0(\LL)\longrightarrow \End_{\Res(\K,k)}(\Res(\LL,k)), \quad f\longmapsto \overline{f}_k,$$
	 		\item $\psi_{\alpha}(\End_{\K}^k(\LL))=\Mat_n(\m_{\K}^k)$,
	 		\item the following diagram is commutative
	 		$$
	 		\xymatrix{
	 			0 \ar[r] & \m_{\LL}^k[G] \ar[d]^{\varphi_k} \ar[r] & \mO_{\LL}[G] \ar[r] \ar[d]^{\varphi_0} & \Res(\LL,k)[G] \ar[r] \ar[d]^{\overline{\varphi}_k} & 0 \\
	 			0 \ar[r] & \End_{\K}^k(\LL) \ar[r] \ar[d]^{\psi_\alpha} & \End_\K^0(\LL) \ar[d]^{\psi_\alpha} \ar[r] & \End_{\Res(\K,k)}(\Res(\LL,k)) \ar[r] \ar[d]^{\psi_{\overline{\alpha}}}& 0 \\
	 			0 \ar[r] & \Mat_n(\m_{\LL}^k)  \ar[r] & \Mat_n(\mO_{\K}) \ar[r]  & \Mat_n(\Res(\K,k)) \ar[r]  & 0 
	 		}
	 		$$
	 		where all the rows are exact and the vertical maps are isomorphisms.
	 	\end{enumerate}
	 \end{theorem}
	 
	 \begin{proof}
	 	(1) Straightforward computation. (2) This is clear from the definition of $\End_\K^{k}(\LL)$ and the fact that $\varphi$ is a $\K$-algebra homomorphism.
	 	(3) It follows from the properties of an integral basis.
	 	(4) The first maps in the columns of the diagram are isomorphisms thanks to Proposition \ref{prop:iso-phi}. The maps $\psi_{\alpha}$ and $\psi_{\overline{\alpha}}$ are also isomorphisms by definition.  Moreover, the rows are exact as a consequence of (1) and (2). The left square is easily shown to be commutative from the direct definitions of $\varphi_0$ and $\varphi_k$.
	 	The right diagram is commutative as a consequence of the commutativity of the left square and the exactness of the rows.
	 \end{proof}
	 
	 \noindent
	 We remark that, for each nonnegative integer $k$ and thanks to the last result, $\m_{\LL}^k[G]$ and $\End_{\K}^k(\LL)$ are ideals of $\mO_{\LL}[G]$ and $\End_{\K}^0(\LL)$, respectively. Moreover, we observe that, as a consequence of Proposition \ref{prop:iso-phi}, the $\mO_\K$-algebra $\mO_{\LL}[G]$ is isomorphic to $\End_{\mO_\K}(\mO_\LL)$ via the following diagram
	 $$
	 \xymatrix{
	 	\End_{\K}(\LL)  & & \ar[ll]_{f\mapsto 1\otimes f} \End_{\mO_{\K}}(\mO_{\LL}) \\
	 	\End_{\K}^0(\LL) \ar[u]^{\subseteq} \ar[rru]_{f\mapsto f_{|\mO_{\LL}}} &&
	 }
	 $$
	 
	 \begin{remark}
	 	The commutative diagram of Theorem \ref{thm:compatibility}(4) justifies the fact that working in $\mO_{\LL}[G]$ or in $\Mat_n(\mO_\K)$ is equivalent, provided that we choose an integral basis $\alpha$ for representing the endomorphisms of $\LL$ as matrices. More specifically, when we want to understand what happens in the  rings $\Res(\K,k)$, we can either consider the ring $\mO_{\LL}[G]$ modulo $\m_{\LL}^k[G]$, or quotient $\Mat_n(\mO_{\K})$ modulo $\Mat_n(\m_{\K}^k)$.
	 \end{remark}

	 \subsection{Rank drops and expansions}\label{sec:rank-drops}
	 
	 The aim of this section is to discuss the drop in rank of $f\in \mO_{\LL}[G]$ to $\overline{f}\in\overline{\LL}[G]$, or more generally to $\Res(\LL,k)[G]$, in terms of cokernel of $f:\mO_{\LL}\rightarrow \mO_{\LL}$ or of a \emph{$\pi$-adic expansion}.
	 
	 Let $\cor{R}\subseteq \mO_{\LL}$ be a \emph{set of representatives} for $\overline{\LL}$, in other words $0\in\cor{R}$ and $\cor{R}$ is mapped bijectively onto $\overline{\LL}$ via the canonical projection $\mO_{\LL}\rightarrow\overline{\LL}$. Among the possible choices of $\cor{R}$, there is a unique one for $\cor{R}\setminus\{0\}$ to be a multiplicative subgroup of $\mO_\LL^\times$ \cite[\S~I.7]{MR1915966}. Such a set of representatives is called \emph{multiplicative}, also known as
	 \emph{Teichm\"uller representatives}, and has the property that $\cor{R}\setminus\{0\}$ is isomorphic to $\overline{\LL}^\times$. We assume, until the end of the present section, that $\cor{R}$ is multiplicative.
	 Thanks to \cite[Cor.~I.5.2]{MR1915966} and the fact that $\mO_{\LL}[G]$ is a free $\mO_{\LL}$-module with basis $G$, we have that 
	 \[
	 \cor{R}[G]=\left\{\sum_{g\in G}r_gg : r_g\in \cor{R}\right\}
	 \]
	 is a set of representatives of $\overline{\LL}[G]$ in $\mO_{\LL}[G]$. In particular,  for every $f\in\mO_{\LL}[G]$, there is a unique \emph{$\pi$-adic expansion} 
	 \begin{equation}\label{eq:exp}
	 	f=\sum_{i\geq 0}f^{(i)}\pi^i
	 \end{equation}
	 of $f$ such that each $f^{(i)}$ is an element of $\cor{R}[G]$. 
  
	 \begin{definition}
	 	Let $R$ be a commutative ring and let $M$ be a finitely generated $R$-module. Then the \emph{inner rank} $\rki(M,R)$ 
	 	of $M$ is the minimum number of generators of $M$ as an $R$-module, in symbols
	 	\[
	 	\rki(M,R)=\min\left\{|X| : X\subseteq M, \ M=\gen{X}_R=\sum_{x\in X}Rx\right\}.
	 	\]
	 	The \emph{inner rank} of $f\in\End_R(M)$ is
	 	$\rki(f,R)=\rki(f(M),R)$.
	 \end{definition}

	 \noindent
	 We observe that, whenever $R$ is a field and $f\in\End_R(M)$, then 
	 $$\rki(f,R)=\rk_R(f)=\dim_R(M)-\dim_R(\ker f).$$ 
	 In the following result, $\rk$ means $\rk_{\K}$. Moreover, if $k$ is a nonnegative integer and $f\in\mO_{\LL}[G]$ is viewed as an element of $\End_{\mO_{\K}}(\mO_{\LL})$, then $\overline{f}_k$ denotes the induced endomorphism of $\Res(\LL,k)$
	 and $f_{|k}$ denotes the truncated sum $\sum_{i=0}^kf^{(i)}\pi^i$, where $f$ is written as in \eqref{eq:exp}.
	 \begin{proposition}\label{prop:cokernel}
	 	Let $k$ be a positive integer and let $f\in\mO_{\LL}[G]$ viewed as an element of $\End_{\mO_{\K}}(\mO_{\LL})$ or of $\End_{\K}(\LL)$. Let, moreover, 
	 	\[
	 	(e_1) \supseteq (e_2) \supseteq \ldots \supseteq (e_n)
	 	\]
	 	denote the elementary divisors of the $\mO_{\K}$-module $\coker(f)=\mO_{\LL}/f(\mO_{\LL})$. Then the following hold: 
	 	\begin{enumerate}[label=$(\arabic*)$]
	 		\item $\rki(\overline{f}_k, \Res(\K,k))=
	 		|\{i : \pi^{k-1}\in(e_i)\}|=n-|\{i : e_i\in(\pi^k)\}|$.
	 		\item $\rki(f, \mO_\K)=\rk f$ and $f(\mO_{\LL})+\m_{\LL}^{k+1}=f_{|k}(\mO_{\LL})+\m_{\LL}^{k+1}$,
	 		\item $\rki(\overline{f}_{k+1},\Res(\K, k+1))\leq\rk f_{|k}$ and both \[\rki(\overline{f}_{k+1},\Res(\K, k+1))<\rk f_{|k} \textup{ or }\rki(\overline{f}_{k+1}, \Res(\K, k+1))=\rk f_{|k}\] can occur.
	 	\end{enumerate}
	 \end{proposition}
	 
	 \begin{proof}
	 	(1) Thanks to Proposition \ref{prop:corresp-phi}(2) the ring $\mO_{\LL}$ is a free module over the principal ideal domain $\mO_{\K}$. Thanks to the assumptions there exists thus a basis $(b_1,\ldots,b_n)$ of $\mO_{\LL}$ over $\mO_{\K}$ such that 
	 	\[
	 	f(\mO_{\LL})=\mO_\K e_1b_1\oplus\ldots\oplus\mO_\K e_nb_n.
	 	\]
	 	Modulo $\m_\K^k$, it follows therefore that a minimal generating set for $\overline{f}_k(\Res(\LL,k))$ as a $\Res(\K,k)$-module is given by $\{
	 	e_ib_i : e_i\notin \m_{\K}^k
	 	\}$, equivalently
	 	\[
	 	\rki(\overline{f}_k, \Res(\K,k))=|\{i : e_i\notin \m_\K^k\}|=n-|\{i : e_i\in\m_\K^k\}|=
	 	|\{i : \pi^{k-1}\in(e_i)\}|
	 	\]
	 	(2) It clearly follows from the definition of $f_{|k}$ and the fact that $f$ belongs to $\mO_{\LL}[G]$ that $f(\mO_{\LL})+\m_{\LL}^{k+1}=f_{|k}(\mO_{\LL})+\m_{\LL}^{k+1}$. We now show the equality $\rki(f,\mO_\K)=\rk f$ holds. Let $b_1,\ldots, b_n$ be as in the proof of (1). Then, by extension of scalars, it follows that 
	 	\[
	 	f(\LL)=\K e_1b_1\oplus\ldots\oplus\K e_nb_n.
	 	\]
	 	In particular, we have that
	 	\[
	 	\rki(f,\mO_\K)=|\{
	 	i : e_i\neq 0\}|=\rk f.
	 	\]
	 	(3) It follows from (2) that \[\rki( \overline{f}_{k+1},\Res(\K,k+1)) = \rk(\overline{(f_{|k})}_{k+1}, \Res(\K,k+1)) \leq \rk f_{|k}.\]
	 	We now show that both equality and strict inequality can occur. For the equality, we can for instance take $f$ to be the identity map: then $f=f_{|k}$ and $\overline{f}_{k+1}$ is the identity on $\Res(\LL,k+1)$. We now give an example for the strict inequality. Assume $\pi=p$ and $p\equiv 3 \bmod 4$. Assume moreover, that $n=2$ and that $\LL$ is the splitting field of $F(X)=X^2+1$ over $\K$. Call $\pm \delta$ the roots of $F$ in $\LL$ and let $\sigma$ be the generator of the Galois group $G$, which swaps $\delta$ and $-\delta$.
	 	Write $\sigma^2=1_{\LL}$. We define $f\in\mO_{\LL}$ to be 
	 	\[
	 	f=1_{\LL}+(1+\delta p)\sigma = 1_{\LL}+\sigma +p(\delta\sigma)=f_{|1}.
	 	\]
	 	With respect to the basis $(1,\delta)$ then the endomorphism $f$ can be expressed as 
	 	\[
	 	f=\begin{pmatrix}
	 		2 & p \\ p & 0
	 	\end{pmatrix}, \textup{ yielding } \rk f =2.
	 	\]
	 	However, modulo $p^2$, the vectors $(2,p)$ and $(p,0)$ are linearly dependent and so \[1=\rki (\overline{f}_2,\Res(\K,2))<\rk f_1=2.\]
	 \end{proof}
	 
	 \begin{remark}\label{rem:ranks}
	 	Let $k$ be a positive integer. Notice that, thanks to Proposition \ref{prop:cokernel}(1) we have that 
	 	\[
	 	|\{i : \pi^{k-1}\in(e_i)\}|=\rki(\overline{f}_k,\Res(\K,k))\leq \rki (\overline{f}_{k+1},\Res(\K,k+1))
	 	\]
	 	and therefore, as a consequence of Proposition \ref{prop:cokernel}(3), we have that 
	 	\[
	 	\rki(\overline{f}_k,\Res(\K,k))\leq \min\{ \rk f_{|\ell} : \ell \geq k-1\}. 
	 	\] 
	 	In particular, 
	 	there exists a nonnegative integer $\ell$ such that, for all $j\geq \ell$, one has that
	 	\[
	 	\rk f=\rki(\overline{f}_{\ell},\Res(\K,\ell))\leq \rk f_{|\ell-1}.
	 	\]
	 	In other words, the ranks of the truncated sums provide upper bounds for the ranks of $f$ over the residue fields and, in the limit, to the rank of $f$. Whether the ranks of $f_{|k}, f_{|k+1},f$ can be related via 
	 	\[
	 	\rk f_{|k} \leq  \rk f_{|k+1}\leq \rk f
	 	\]
	 	is not a priori clear. It seems reasonable to expect that, in the context of the last inequalities, $k=0$ might play a special role.
	 \end{remark}

	 \subsection{Module endomorphisms}\label{sec:rank_polynomials}
	 The following assumptions will be valid until the end of the present section. 
	 Let $S/R$  be a $G$-Galois extension of local principal ideal rings.
	 In other words, there exists a $G$-Galois extension of rings $\tilde{S}/\tilde{R}$
	 such that the following properties hold: 
	 \begin{enumerate}[label=$(\arabic*)$]
	 	\item the rings $\tilde{R}$ and $\tilde{S}$ are local principal ideal domains, equivalently DVRs,
	 	\item the rings $R$ and $S$ are quotient rings of $\tilde{R}$ and $\tilde{S}$, respectively, that is, there exist ideals $I_R$ of $\tilde{R}$ and $I_S$ of $\tilde{S}$ such that $R=\tilde{R}/I_R$ and $S=\tilde{S}/I_S$ \footnote{Equivalently, either $I_R=I_S=0$ or there exists $k$ such that $R=\tilde{R}/\tilde{R}\omega^k$ and $S=\tilde{S}/\tilde{S}\omega^k$. Note that the case when the ideals are non-trivial covers the case of all finite commutative chain rings; cf.\ \cite{hou2004enumeration}},
	 	\item the following natural diagram is commutative. 
	 	\[
	 	\xymatrix{
	 		\tilde{R} \ar[d] \ar[r] & \tilde{S} \ar[d] \\
	 		R \ar[r] & S
	 	}
	 	\]
	 	
	 \end{enumerate}
	 In particular, if the maximal ideal $\m_{\tilde{R}}$ of $\tilde{R}$ is equal to $\tilde{R}\omega$, then, for each nonnegative integer $k$, we have that $\m_{\tilde{R}}^k=\tilde{R}\omega^k$.  
	 Moreover, thanks to Proposition \ref{prop:iso-phi}(2), the modules $S$ and $\tilde{S}$ are free over $R$ and $\tilde{R}$, respectively.
	 For each nonnegative integer $i$, we denote by $\Res(R,i)$ the quotient ring of $R$ by the $i$-th power of its maximal ideal $\m_R$ and we set $\overline{R}=\Res(R,1)$. 
	 
	 We assume, additionally, that the ring extension $S/R$ is \emph{unramified}, i.e.\ the ring $\tilde{S}$ is local with maximal ideal $\m_{\tilde{S}}=\tilde{S}\omega$. Observe that for $S/R$ to be unramified is equivalent to saying that the extension of fields $\mathrm{Frac}(\tilde{S})/\mathrm{Frac}(\tilde{R})$ is unramified. Write $\overline{S}$ for the residue field of $S$ and write, additionally, $\m_S$ and $\Res(S,i)$ for the analogues of $\m_R$ and $\Res(R,i)$ for $S$, respectively.
	 Then the 
	 extension of the residue fields $\overline{S}/ \overline{R}$ is $G$-Galois. We use the bar notation for elements and $\overline{R}$-spaces in $\overline{S}$.
	 To conclude, for an element $f\in S[G]$, we write 
	 \[
	 \val(f)=\max\{k : f\in \m_S^k[G]\}\in \Z_{\geq 0}\cup\{+\infty\}.
	 \]
	 
	 \begin{definition}
	 	Let $M$ be a finitely generated $R$-module $M$. Then the \emph{free rank} $\frk(M,R)$ of $M$ is the largest among the ranks of free $R$-submodules of $M$, in symbols
	 	\[
	 	\frk(M,R)=\max\{\rki(N,R) : N \textup{ is a free $R$-submodule of }M\}.
	 	\]
	 \end{definition}
	 
	 \begin{lemma}\label{lem:rank_nullity}\label{lem:free_ker}
	 	Let $f \in S[G]$. Then the following hold:
	 	\begin{enumerate}[label=$(\arabic*)$]
	 		\item one has $\rki(f,R) +\frk(\ker(f),R)=n.$ 
	 		\item there exists a free $R$-submodule $M$ of $\ker (f)$ such that 
	 		$$\frk(M,R)=\dim_{\overline{R}}(\overline{M})=\frk(\ker(f),R).$$
	 	\end{enumerate}
	 \end{lemma}
	 
	 \begin{proof}
 (1) In view of \cref{prop:iso-phi}, we consider $f$ as an element of $\End_R(S)$ and so, thanks to the isomorphism theorems, we have an induced $R$-module isomorphism $S/\ker(f)\rightarrow f(S)$. This means in particular that $\rki(S/\ker(f),R)=\rki(f(S),R)$. It follows now from the elementary divisor theorem that $\rki(S/\ker(f),R)=n-\frk(\ker(f),R)$ and so the statement is proven. 

Alternatively, one can observe that $f$, as an element of $\End_R(S)$, can be represented as a square matrix over $R$ in Smith normal form, whose diagonal is $(e_1,\ldots,e_n)$, after choosing any $R$-basis of $S$. This is ensured by the fact that Smith normal forms exist over any principal ideal ring; see e.g. \cite[Theorem 12.3 and  subsequent remark]{kaplansky1949elementary}. The claim follows from the fact that $\rki(f,R)=|\{i : e_i\neq 0 \}|$, which can be deduced in an analogous way to what shown in Proposition \ref{prop:cokernel}(1) for $R=\mO_\K$ and $S=\mO_\LL$.

   (2) Thanks to the definition of free rank of $\ker(f)$, there exists a free $R$-sumbodule $M$ of $\ker(f)$ such that $\rki(M,R)=\frk(\ker(f),R)$. The submodule $M$ being free over $R$, one also has $\rki(M,R)=\dim_{\overline{R}}(\overline{M})$ and so the proof is complete.
	 \end{proof}

	 \begin{proposition}\label{prop:innerrank_lowerbound}
	 	Let $f \in S[G]$ be a nonzero $\sigma$-polynomial. Then, one has
	 	$$\rki(f,R) \geq n-\deg_\sigma(f).$$
	 \end{proposition}
	 
	 \begin{proof}
	 	We work by induction on $v=\val(f)$. We start by assuming that $v=0$ and we show
	 	that $\frk(\ker(f),R)\leq \deg_\sigma(f)$. By Lemma \ref{lem:free_ker}(2), there exists a free $R$-submodule  $M\subseteq \ker(f)$ such that $\dim_{\overline{R}}(\overline{M})=\frk(\ker(f),R)$. Clearly, $\overline{M} \subseteq \ker(\overline{f}))$ and so it follows from the field case that
	 	$$ \frk(\ker(f),R)=\dim_{\overline{R}}(\overline{M})\leq \dim_{\overline{R}}(\ker(\overline{f}))\leq \deg_{\sigma}(\overline{f}) \leq \deg_{\sigma}(f).$$
	 	Lemma \ref{lem:rank_nullity} yields
	 	$\rki(f,R)=n-\frk(\ker(f),R)\geq n-\deg_\sigma(f)$.
	 	
	 	Assume now that $v=\val(f)>0$ and that the claim holds for any $\sigma$-polynomial of lower valuation. Let $h\in S[G]$ be such that 
	 	$f=\omega h$ and observe that such $h$ exists, because $v>0$. Now, $f$ and $h$ have the same degree and the following inequalities hold, by definition of inner rank and the induction hypothesis: 
	 	\[
	 	\deg_\sigma(f)=\deg_{\sigma}(h)\geq n-\rki(h,R)\geq n-\rki(f,R).
	 	\]
	 	In particular, we have that $\rki(f,R)\geq n-\deg_\sigma(f)$.
	 \end{proof}
	 
	 \noindent
	 The analogues of the following two results in the context of finite fields can be derived from results included in \cite{MR1503119}.
	 
	 \begin{proposition}\label{prop:annihilator_recursive}
	 	Let $M$  be a free $R$-submodule of $S$ of rank $r\leq n$ such that $\dim_{\overline{R}}(\overline{M})=r$ and let $(\beta_1,\ldots,\beta_r)$ be an $R$-basis of $M$.  Then, there exists a unique monic $f_M \in S[G]$ of $\sigma$-degree $r$ such that $f_M(M) = \{ 0 \}$. Moreover, the $\sigma$-polynomial $f_M$ is defined by induction as:
	 	\begin{equation}\label{eq:annihilator_poly}
	 		f_M= \begin{cases} \mathrm{id} & \mbox{ if } r=0,\\
	 			\left(\sigma- \frac{\sigma(f_{N}(\beta_r))}{f_{N}(\beta_r)}\mathrm{id}\right)
	 			\circ f_{N} & \mbox{ if } r\geq1,
	 		\end{cases}
	 	\end{equation}
	 	where $N=\langle \beta_1,\ldots, \beta_{r-1}\rangle_R$.
	 \end{proposition}
	 
	 \begin{proof}
	 	Let $f$ be defined as in \eqref{eq:annihilator_poly}. First, we show that $f$ is well-defined, that is, at each step we can  divide by $f_N(\beta_r)$. Suppose, for a contradiction, that this is not the case. This means that $f_N(\beta_r)\in \m_S$ and, thus, that $\overline{f}_N(\overline{\beta}_r)=0$. It follows that $\overline{f}_N$ has $\sigma$-degree $r-1$ and it vanishes on $\overline{M}$, which has $\overline{R}$-dimension $r$ by assumption. This yields a contradiction, and shows that $f_M$ defined in \eqref{eq:annihilator_poly} is well-defined. Moreover, by definition it is easy to see that $f(\beta_i)=0$ for every $i\in\{1,\ldots,r\}$. 
	 	
	 	Now, suppose that there exist two monic $\sigma$-polynomials $f_1,f_2$ of $\sigma$-degree $r$ such that $f_1(M)=f_2(M)=0$. Then $f_1-f_2$ is either $0$ or has $\sigma$-degree at most $r-1$ and $(f_1-f_2)(M)=\{0\}$. However, this second case is not possible because it contradicts Proposition \ref{prop:innerrank_lowerbound}. Thus, $f_1=f_2$, which shows the uniqueness of $f_M$. 
	 \end{proof}
	 
	 \noindent
	 The $\sigma$-polynomial $f_M$ defined in \eqref{eq:annihilator_poly} is called the \emph{annihilator polynomial} of the free $R$-submodule $M$.
	 
	 \begin{corollary}\label{cor:monic_generator}
	 	Let $M$  be a free $R$-submodule of $S$ of rank $r$ such that $\dim_{\overline{R}}(\overline{M})=r$ and let $(\beta_1,\ldots,\beta_r)$ be an $R$-basis of $M$. Let $f\in S[G]$ be a $\sigma$-polynomial such that $f(M)=0$. Then,
	 	there exists $h\in S[G]$, such that $f=h\circ f_M$.
	 \end{corollary}
	 
	 \begin{proof}
	 	Since $f_M$ is monic, we can perform the right Euclidean algorithm and perform right division of $f$ by $f_M$; see for instance \cite{MR2006329}. Let $h, h' \in S[G]$ be such that $f=h\circ f_M+h'$ and either $h'=0$ or $\deg_{\sigma}(h')<r$. Since $f(M)=f_M(M)=\{0\}$, we have $h'(M)=\{0\}$ and Proposition \ref{prop:innerrank_lowerbound} yields $h'=0$.
	 \end{proof}

	 \noindent
	 We now define a truncated version of the Moore matrix as a rectangular  submatrix of the $G$-Moore matrix. For a positive integer $r$ and a vector $\alpha=(\alpha_1,\ldots \alpha_r)\in S^r$, we define the $s\times r$ matrix $\MG_{s,\sigma}(\alpha)$ as the matrix
	 $$ \MG_{s,\sigma}(\alpha)=(\sigma^i(\alpha_j))_{0\leq i \leq s-1,\, 1\leq j \leq r}. $$
	 For a module $M$ over $R$, we write \emph{elementary divisors} of $M$ meaning the images in $R$ of the elementary divisors of $M$ as a $\tilde{R}$-module. This is the same as taking the elementary divisors coming from the Smith normal form of $M$.
	 
	 \begin{proposition}\label{prop:restricted_moore}
	 	Let {$r\leq n$}, $\alpha=(\alpha_1,\ldots \alpha_r)\in S^r$, and  $(e_1)\supseteq \ldots \supseteq (e_n)$  be the elementary divisors of the $R$-submodule $\langle \alpha_1,\ldots,\alpha_r\rangle_R$ of $S$. Then there exist $A\in \GL(r,R)$ and
	 	$\beta\in S^r$ such that 
	 	$\MG_{r,\sigma}(\beta)\in \GL(r,S)$ and the following equality holds: 
	 	$$ \MG_{r,\sigma}(\alpha)A=\MG_{r,\sigma}(\beta)\diag(e_1, \ldots, e_r).$$
	 	Moreover, $\overline{\alpha}_1, \ldots,\overline{\alpha}_r$ are $\overline{R}$-linearly independent if and only if 
	 	$\MG_{r,\sigma}(\alpha)\in \GL(r,S)$. 
	 \end{proposition}
	 
	 \begin{proof}
	 Set $M=\langle \alpha_1,\ldots,\alpha_r\rangle_R$ and note that 
	  $e_{r+1},\ldots,e_n$ are all equal to $0$, because $M$ has {inner} rank at most $r$.
	 	Let $\beta=(\beta_1,\ldots \beta_r)\in S^r$ be a vector such that $\langle\beta_1,\ldots,\beta_r\rangle_R$ is a free $R$-submodule of rank $r$ of $S$ and 
	 	$$M=\langle \alpha_1,\ldots,\alpha_r\rangle_R=\langle e_1\beta_1,\ldots,e_r\beta_r\rangle_R.$$  
	 	Set $s=\max\{i=1,\ldots, r : e_i \neq 0\}$, so that 
	 	 $s=\rki(M,R)$ and $M=\langle e_1\beta_1,\ldots,e_s\beta_s\rangle_R$. Then, there exist $A_1 \in \Mat_{s\times r}(R)$, $A_2 \in \Mat_{r\times s}(R)$  such that 
	 	 $$(e_1\beta_1,\ldots,e_s\beta_s)A_1= \alpha \ \textup{ and }\ \alpha A_2=(e_1\beta_1,\ldots,e_s\beta_s).$$ 
	 	 From the above equalities, we derive that $\diag((e_i)_{i=1,\ldots,s})A_1A_2=\diag((e_i)_{i=1,\ldots,s})$ and hence, for each $i\in\{1,\ldots,s\}$, there exists $v_i\in \Ann_R(e_i)$ with the property that 
	 	 $$A_1A_2=I_s+\diag((v_i)_{i=1,\ldots,s})D,$$
	 	 for some $D \in \Mat_s(R)$.
	 	 In particular, $A_1A_2\in\GL(s,R)$ and therefore $A_2$ admits a left inverse. At this point, consider the $R$-module $\Ann_{R^r}(\alpha)$. By definition, $\Ann_{R^r}(\alpha)$ coincides with the kernel of the map $\phi:R^r\rightarrow S$ given by $\phi(u)= u\alpha^\top$. Since the image of $\phi$ is $\langle \alpha_1,\ldots,\alpha_r\rangle_R$, we deduce from  \cref{lem:free_ker}(1) that $\frk(\Ann_{R^r}(\alpha),R)=r-s$. Moreover, thanks to \cref{lem:free_ker}(2), we choose a free $R$-submodule $N$ of $\Ann_{R^r}(\alpha)$ with the property that $\dim_{\overline{R}}(\overline{N})=r-s$. If we select a basis of $N$ and fill its elements as columns of a matrix $C_2\in \Mat_{r \times (r-s)}(R)$, we obtain that the matrix
	 	 $$ A=\left(\begin{array}{c|c} A_2 &
	 	  C_2 \end{array}\right)$$
	 	  belongs to $\GL(r,R)$ and 
	 	 satisfies $\alpha A=(e_1\beta_1,\ldots,e_r\beta_r)$.
	 	Since the entries of $A$ are fixed by $\sigma$, we have
	 	$$ \MG_{r,\sigma}(\alpha)A=\MG_{r,\sigma}(e_1\beta_1,\ldots,e_r\beta_r)=\MG_{r,\sigma}(\beta)\diag((e_i)_{i=1,\ldots,r}).$$
	 	The $R$-module $\langle\beta_1,\ldots,\beta_r\rangle_R$ is free of rank $r$ with elementary divisors all equal to $1$. Hence, $\overline{\beta}_1,\ldots,\overline{\beta}_r$ are $\overline{R}$-linearly independent and this holds if and only if $\MG_{r,\sigma}(\overline{\beta})$ belongs to $ \GL(r,\overline{S})$; see e.g.\ \cite[Cor.~4.13]{lam1988vandermonde}. This last property is in turn equivalent to $\MG_{r,\sigma}(\beta)$ belonging to $\GL(r,S)$ and this concludes the proof.
	 \end{proof}
	 
	 \noindent
	 In the following result, we give an explicit expression for the annihilator polynomial of a free $R$-submodule of $S$. {For the analogue result over finite fields see for instance \cite{Sheekey/16}.}
	 
	 \begin{proposition}\label{prop:explicit_annihilator}
	 	Let $M$  be a free $R$-submodule of $S$ {of rank $r\leq n$} such that $\dim_{\overline{R}}(\overline{M})=r$ and let $\beta=(\beta_1,\ldots,\beta_r)$ be an $R$-basis of $M$. 
	 	Then the annihilator polynomial of $M$ is given by
	\begin{equation}\label{eq:annihilator_formula}f_M=h_r^{-1}\sum_{i=0}^r(-1)^{r-i}h_i\sigma^i,\end{equation}
	 	where $h_i$ is the determinant of the matrix obtained by removing the $(i+1)$-th row from $\MG_{r+1,\sigma}(\beta)$.
	 \end{proposition}
	 
	 \begin{proof}
	 	Let $f \in S[G]$ be the $\sigma$-polynomial $f=\sum_i(-1)^{r-i}h_i\sigma^i$ and let $\alpha\in S$. Then, for $v_\alpha=(\alpha,\sigma(\alpha),\ldots,\sigma^r(\alpha))^\top$, it is easy to check that 
	 	$$ f(\alpha)=\det( \MG_{r+1,\sigma}(\beta) \mid v_\alpha).$$
	 	From Proposition \ref{prop:restricted_moore} we derive that $f(\alpha)=0$ if and only if $\alpha \in M$.
	 	Moreover, the leading coefficient of $f$ is $h_r=\det(\MG_{r,\sigma}(\beta))$. Thanks again to Proposition \ref{prop:restricted_moore}, $\MG_{r,\sigma}(\beta)$ is invertible and hence $h_r=\det(\MG_{r,\sigma}(\beta))$ is a unit. Thus, $h_r^{-1}f$ belongs to $S[G]$, is monic and $\ker(h_r^{-1}f)=\ker(f)=\ker(f_M)$.  The statement follows from Proposition \ref{prop:annihilator_recursive}.
	 \end{proof}
	 
	 \noindent
	 In a similar way to what is done in \cite[Lem.~3]{Sheekey/16}, the next result gives a necessary condition for $\sigma$-polynomials to have maximum kernel. Denote by $\Norm{S}{R}(\alpha)$ the norm of an element $\alpha \in S$ with respect to the ring extension $S/R$, that is
	  $$\Norm{S}{R}(\alpha)=\prod_{g\in G} g(\alpha) = \prod_{i=0}^{n-1} \sigma^i(\alpha). $$
	 
	 \begin{theorem}\label{th:norm}
	 	Let $\ell\leq n$ and  $f=f_0\mathrm{id}+\ldots+f_\ell\sigma^\ell\in S[G]$ be  such that $\rki(f)=n-\ell$. Then one has
	 	$\Norm{S}{R}(f_0)=(-1)^{\ell n}\Norm{S}{R}(f_\ell)$. 
	 \end{theorem}
	 
	 \begin{proof}
    	 	It follows from the hypotheses and Lemma \ref{lem:rank_nullity}(1) that $\frk(\ker(f),R)=\ell$. Thanks to Lemma \ref{lem:free_ker}(2) there exists   a free $R$-submodule $M$ of $\ker(f)$ of rank $\ell$ such that $\dim_{\overline{R}}(\overline{M})=\ell$. \
       By Corollary \ref{cor:monic_generator}, there exists $F\in S[G]$ such that $f=F\circ f_M$. Since $\deg_\sigma f=\deg_\sigma f_M$ and $f_M$ is monic, $F$ has $\sigma$-degree $0$, i.e. there exists $\lambda \in S$ such that $f=\lambda f_M$. Thus, $f_0=\lambda (f_M)_0$ and $f_\ell=\lambda (f_M)_\ell=\lambda$. Let $\beta=(\beta_1,\ldots,\beta_\ell)$ be an $R$-basis of $M$, and let 
       $h_i$ be the determinant of the $\ell\times \ell$ submatrix of $M_{\ell+1,\sigma}(\beta)$ obtained by removing the $(i+1)$-th row, for $i\in\{0,\ldots,\ell\}$.
        In particular, we can observe that 
        $$ h_0=\det(((\sigma^i(\beta_j))_{1\leq i\leq \ell,1\leq j\leq \ell})=\sigma(\det(((\sigma^i(\beta_j))_{0\leq i\leq \ell-1,1\leq j\leq \ell}))=\sigma(h_\ell). $$
	 	 From Proposition \ref{prop:explicit_annihilator} we obtain
	 	$$f_0=\lambda (f_M)_0=\lambda h_\ell^{-1}(-1)^\ell h_0=(-1)^\ell\lambda h_\ell^{-1}\sigma(h_\ell) =(-1)^\ell f_\ell\sigma(h_\ell)h_\ell^{-1}.$$
	 	Therefore, we deduce
	 	$$ \Norm{S}{R}(f_0)=(-1)^{\ell n}\Norm{S}{R}(f_\ell).$$
	 \end{proof}

	 \subsection{Rank-metric codes over local principal ideal rings}\label{sec:gabidulin}
	 Rank-metric codes were introduced by Delsarte in \cite{MR514618} and by Gabidulin in \cite{MR791529} only for a combinatorial interest. They were originally thought as either sets of matrices over a finite field or as sets of vectors over an extension field. However, they can be viewed in many equivalent ways. In the case of square matrices and when the field of definition $\mathbb E$ admits a $G$-Galois extension $\F$, then it was shown that sets of matrices can be equivalently represented as subsets of $\F[G]$ and hence also rank-metric codes; see e.g.\ \cite{ACLN20+}. Here, the rank induces a metric, called the \emph{rank-metric}. Concretely, we can actually extend this to the case of rings. For this, let $S/R$ be a $G$-Galois extension. Then 
	 the distance function is the map $d:S[G]\times S[G] \rightarrow \Z_{\geq 0}$, given by
	 $$(f,g) \longmapsto d(f,g)=\rki(f-g,R). $$
	 Until the end of the present section, we restrict ourselves to the same setting from Section \ref{sec:rank_polynomials}, that is $S/R$ is an unramified $G$-Galois extension of local principal ideal rings.
	 
	 \begin{definition}
	 	A \emph{rank-metric code} $\Ccal$ is a subset of $S[G]$ endowed with the rank distance. Moreover, the \emph{minimum rank distance} of $\Ccal$ is the integer
	 	$$ d(\Ccal)= \min \{ d(f,g) : f,g \in \Ccal, f\neq g\}.$$
	 	If $\Ccal$ is an $R$-submodule of $S[G]$, then the code $\Ccal$ is said to be \emph{$R$-linear}.
	 \end{definition}
	 
	 \noindent
	 Observe that for an $R$-linear rank-metric code $\Ccal$ we have
	 $$ d(\Ccal)= \min \{ \rki(f,R) : f\in \Ccal\setminus\{0\}\}.$$
	 The following result can be obtained by generalizing, in a straightforward way, the techniques from \cite{MR514618,MR791529,MR1306980,MR4038895}.
	 
	 \begin{proposition}[Singleton-like bound]\label{prop:SBound}
	 	Let $\Ccal\subseteq S[G]$ be an $R$-linear rank-metric code. Then the following inequality holds
	 	\begin{equation}\label{eq:Singleton} 
	 		\rki(\Ccal,R)\leq n(n-d(\Ccal)+1).
	 	\end{equation}
	 	Moreover, if $S$ is finite and $\Ccal\subseteq S[G]$ is any (non-necessarily linear) rank-metric code, then
	 	\begin{equation}\label{eq:Singleton2}
	 		\log_{|R|}|\Ccal|\leq \rki(\Ccal,R)\leq n(n-d(\Ccal)+1). 
 		\end{equation}
	 \end{proposition}
	 
	 \noindent
	 We remark that, for finite $S$,  the equality $\log_{|R|}|\Ccal|= \rki(\Ccal,R)$ is achieved precisely when $\Ccal$ is a free $R$-submodule of $S[G]$. The following definition generalizes that of rank-metric codes over fields.
	 
	 \begin{definition}
	 	A rank-metric code $\Ccal\subseteq S[G]$ is called \emph{maximum rank distance} (MRD), if it is a free $R$-submodule such that \eqref{eq:Singleton} is met with equality.
	 \end{definition}
	 
	 \noindent
	 We remark that, by their definition, rank-metric codes are necessarily linear. Some authors include non-linear codes meeting the bound of \eqref{eq:Singleton2} with equality in the family of MRD codes, but this is not the case here. 
	 
	 The following result shows that, when $R$ is finite, the structural parameters of a rank-metric code $\Ccal$ are completely determined by its residue image $\overline{\Ccal}$. We remark that \cref{thm:MRDiff} is analogous to a theorem given in \cite{walker} for the case of the Hamming metric, which was extended in \cite{KK} to the case of the rank-metric. In both works the equality $d(\Ccal)\leq d(\overline{\Ccal})$ is proven.
	 
	 	 \begin{theorem}\label{thm:MRDiff}
	  Assume that $R$ is finite and let $\Ccal$ be a free $R$-module of $S[G]$. Then $$d(\Ccal)=d(\overline{\Ccal}) \ \ \textup{ and } \ \  \rki(\Ccal,R)=\dim_{\overline{K}}(\overline{\Ccal}).$$
	    Moreover, $\Ccal$ is MRD in $L[G]$ if and only if $\overline \Ccal$ is MRD in $\overline L[G]$.
	 \end{theorem}
	 \begin{proof} 
The equality $\rki(\Ccal,R)=\dim_{\overline{K}}(\overline{\Ccal})$ follows from the definitions and the fact that $\Ccal$ is free. We now show that $d(\Ccal)=d(\overline{\Ccal})$. To this end, set $t=\min\{ i \in \mathbb N : \pi^i=0 \}$ be the nilpotency index of $R$ and define the function
\[
{\rm wt}:\Ccal\setminus\{0\} \longrightarrow \{0,\ldots,t-1\}, \quad f\longmapsto {\rm wt}(f)=\max\{i : f\in \pi^i\Ccal \}.
\]
We use the last map to define $\varphi: \Ccal\setminus\{0\}\rightarrow \pi^{t-1}\Ccal\setminus\{0\}$ by sending $f$ to $\pi^{t-{\rm wt}(f)-1}f$. 
Note that, for each $f\in\Ccal\setminus\{0\}$, one has 
$\rki(\varphi(f),R)\leq \rki(f,R)$ and so we derive
\[
d(\Ccal)=\min\{\rki(f,R) : f\in \pi^{t-1}\Ccal\setminus\{0\}\}.
\]
Defining now $\rho:\overline{\Ccal}\rightarrow\pi^{t-1}\Ccal$ to be the isomorphism of $R$-modules induced by scalar multiplication by $\pi^{t-1}$, we get 
\[
d(\Ccal)=\min\{\rk(\rho^{-1}(f)) : f\in \pi^{t-1}\Ccal\setminus\{0\}\}=d(\overline{\Ccal}).
\]
The last part of the statement readily follows.  
	 \end{proof}

	 \noindent
	 At this point we introduce two families of MRD codes naturally generalizing finite field constructions within our setting. 
	 
	 \begin{definition}
	 	For a fixed positive integer $\ell \leq n$, the \emph{Gabidulin code} $\cor{G}_{\ell,\sigma}$ is 
	 	$$ \mathcal G_{\ell,\sigma}= \left\{ \sum_{i=0}^{\ell-1} f_i \sigma^i : f_i \in S \right\}\subseteq S[G].$$
	 \end{definition}
	 
	 \noindent Gabidulin codes were introduced independently in \cite{MR514618} and in \cite{MR791529} over finite fields, while over general fields they have been first constructed in \cite{MR1306980} and later rediscovered in \cite{MR2503966} and \cite{augot2013rank}. Very recently, this family of codes has also been generalized to finite principal ideal rings in \cite{MR4038895}. In all these settings the authors showed that the constructed codes are MRD. The following result shows that, also in the setting of Galois extensions of local principal ideal rings, Gabidulin codes are MRD.
	 
	 \begin{proposition}\label{prop:Gab_are_MRD}
	 	Let $0<\ell\leq n$ be integers. Then the $R$-module $\mathcal G_{\ell,\sigma}$ is an MRD code.
	 \end{proposition}
	 
	 \begin{proof}
	 	This is a direct consequence of Proposition \ref{prop:innerrank_lowerbound}. 
	 \end{proof}
	 
	 \noindent
	 The following is the second most well-known construction of MRD codes. 
	 
	 \begin{definition}
	 	Let $\ell\leq n$ be a positive integer, $h\in\mathbb Z$, and $\eta \in S\setminus \{0\}$. The module 
	 	$$\mathcal H_{\ell,\sigma}(\eta,h)=\left\{\sum_{i=0}^{\ell-1} f_i \sigma^i + \eta \sigma^h(f_0)\sigma^\ell : f_i \in S \right\}\subseteq S[G] $$
	 	is called a \emph{twisted Gabidulin code}.
	 \end{definition}
	 
	 \noindent This family of codes was recently introduced by Sheekey in \cite{Sheekey/16} 
	 and it is the first general construction of MRD codes after Gabidulin codes.   However, the property of being MRD depends on the element $\eta$, which needs to have a prescribed norm. In Sheekey's work, the construction was only given for the case of $R$ and $S$ being finite fields, but it is not difficult to extend it to general fields.  We now show that this family can be further generalized to MRD codes over local principal ideal rings. 
	 	 \begin{theorem}\label{thm:sheekey_rings}
Let $0<\ell\leq n$ be integers.	 	Let, moreover,  $h\in\mathbb Z$ and $\eta\in S$ satisfy the following:
	 	\begin{enumerate}[label=$(\arabic*)$]
	 	    \item If $R$ is infinite, then $\Norm{S}{R}(\eta)\neq (-1)^{\ell n}$. 
	 	\item If $R$ is finite, then $\Norm{\overline{L}}{\overline{K}}(\overline\eta)\neq (-1)^{\ell n}$. 
	 	\end{enumerate}
	 	Then the $R$-module $\mathcal H_{\ell,\sigma}(\eta,h)$ is an %$R$-linear 
	 	MRD code.
	 \end{theorem}
	 \begin{proof}
	     	Let $f=f_0\mathrm{id}+\ldots+f_\ell\sigma^\ell\in \mathcal H_{\ell,\sigma}(\eta,h)$  be a nonzero $\sigma$-polynomial. We will show that $\rki(f,R)\geq n-\ell+1$. If $\deg_{\sigma}(f)\leq \ell-1$, then, by Proposition \ref{prop:innerrank_lowerbound}, we automatically have $\rki(f,R)\geq n-\ell+1$. If $\deg_{\sigma}(f)=\ell$, assume by contradiction that $\rki(f,R)= n-\ell$. Then Theorem \ref{th:norm}  yields 
	 $$\Norm{S}{R}(f_0)=(-1)^{\ell n}\Norm{S}{R}(f_\ell)=(-1)^{\ell n}\Norm{S}{R}(\eta)\Norm{S}{R}(f_0)$$
	     or, in other words 
	     \begin{equation}\label{eq:norm_proof}
	         \Norm{S}{R}(f_0)(\Norm{S}{R}(\eta)-(-1)^{\ell n})=0.
	     \end{equation}
 If $R$ is infinite, then there are no non-trivial zero divisors in $S$, and thus \eqref{eq:norm_proof} contradicts the hypothesis. If, on the contrary, $R$ is finite, then we use Theorem \ref{thm:MRDiff} combined with the necessary condition of twisted Gabidulin codes over fields to be MRD; see \cite{Sheekey/16}.
	 \end{proof}
	 
\noindent
As we saw with \cref{thm:MRDiff}, when $R$ is finite, a code $\Ccal$ is MRD if and only if its image $\overline{\Ccal}$ over $\overline{K}$ is MRD in the classical sense. This suggests that the more interesting setting to be explored is that of MRD codes defined over (infinite) discrete valuation rings and the variation of their parameters over the associated residue field. We investigate this question in the next section and in the language of affine buildings.

	 \section{A digression into rank-metric codes via Bruhat-Tits buildings}
	 \label{sec-bruhat}
	 
	 \noindent
	 In this section, we study the behaviour of the dimension of a linear space of matrices when passing to the residue field, by making a connection with \emph{Mustafin varieties}. In \cref{pre-brutits}, we outline some classical results in the theory of Bruhat-Tits buildings. Moreover, in \cref{subsec:arithmfound}, we give a short introduction to the arithmo-geometric foundations necessary for our work. In \cref{pre-rat,,sec:mustafin,,subsec:class} we introduce Mustafin varieties and state a classification result from \cite{hahn2020mustafin}. Finally, in \cref{subsec:rankmet}, we connect the theory of rank-metric codes to the notion of convexity in Bruhat-Tits buildings.
	 
	 \subsection{The Bruhat-Tits building}
	 \label{pre-brutits}
	 The theory of Bruhat-Tits buildings originated from the rich theory of reductive algebraic groups \cite{MR788969, MR896503, MR2320966}. In this paper, we focus on the building $\mathfrak{B}_d$ associated to the reductive group $\mathrm{PGL}(V)$ (i.e. the group of linear isomorphisms of $V$ up to scaling), where $V$ is a $d$-dimensional vector space over the discretely valued field $K$. 
	 
	 \noindent
	 We call a \emph{lattice} in $V$ any $\Ocal_K$-submodule of $V$ of rank $d$. For a lattice $\Lambda$ in $V$, we denote by $[\Lambda] =\left\{ c \Lambda \colon c \in K^\times \right\}$ the \emph{homothety class} of $\Lambda$. 
	 The set of homothety classes of lattices is denoted by
	$\Bfrak_d^0 = \{ [\Lambda] \colon \Lambda \textrm{ is a lattice in } V \}$. We say that two lattice classes $[\Lambda_1], [\Lambda_2]$ are \emph{adjacent} if there exists $n \in \Z$ such that
	 \[
	 \pi^{n+1} \Lambda_1 \subseteq \Lambda_2 \subseteq \pi^{n} \Lambda_1.
	 \]
	 The  \emph{Bruhat-Tits} building $\Bfrak_d$ is the infinite flag simplicial complex $\Bfrak_d$, whose set of vertices is $\Bfrak_d^0$ and whose edges are given by the adjacency relation defined above. The simplices in $\Bfrak_d$ are the subsets of $\Bfrak_d^0$ whose elements are pairwise adjacent. For a detailed treatment of buildings, we refer the reader to the book by P. Abramenko and K. Brown \cite{MR2439729}. The Bruhat-Tits building $\Bfrak_d$ is endowed with a notion of convexity given in the following definition.
	 
	 \begin{definition}
	 	Let $\Gamma$ be a set of homothety classes of lattices. Then $\Gamma$ is \emph{convex} if, for any two lattice classes $[\Lambda_1], [\Lambda_2] \in \Gamma$ and any $\tilde{\Lambda}_1 \in [\Lambda_1]$ and ${\tilde{\Lambda}_2} \in  [\Lambda_2]$, we have $[\tilde{\Lambda}_1 \cap \tilde{\Lambda}_2] \in \Gamma$.
	 \end{definition}
	 
	 \noindent
	 The \emph{convex hull} $\mathrm{conv}(\Gamma)$ of a set $\Gamma$ of homothety classes of lattices is the smallest convex set in the building $\mathfrak{B}_d$  containing $\Gamma$. It is known that, when $\Gamma=\{[\Lambda_1],\dots,[\Lambda_n]\}$  is finite, then its convex hull  
	 \begin{equation}
	 	\mathrm{conv}(\Gamma) = \left\{\left[\bigcap_{i=1}^n\pi^{m_i}\Lambda_i\right] \colon m_1,\dots,m_n\in\mathbb{Z}\right\}
	 \end{equation}
	 is also finite (see \cite{MR2367213, zhang2020computing}).
	 
	 \begin{example}[The building $\Bfrak_2(\Q_2)$]
	 	In Figure \ref{fig:buildingb2}, we illustrate a local realization of the Bruhat-Tits building $\Bfrak_2$ over $\mathbb{Q}_2$. In this case, $\Bfrak_2$ is an infinite three-valent tree.

	 	\begin{figure}[ht]
	 		\centering
	 		\resizebox{0.7\textwidth}{!}{
	 			\begin{tikzpicture}
	 				\filldraw[] (0,0) circle [radius=0.1] node[above = 3.75mm] {$\begin{pmatrix}1 & 0 \\ 0 & 1\end{pmatrix}$};
	 				\draw[, thick] (0,0) -- (0, -4);
	 				\draw[, thick] (0,0) -- (3.464, 2);
	 				\draw[, thick] (0,0) -- (-3.464, 2);
	 				\filldraw[] (0,-4) circle [radius=0.1] node[above right = 1mm = 3.75mm] {$\begin{pmatrix}1 & 0 \\ 1 & 2\end{pmatrix}$};
	 				\filldraw[]  (-1.732, -5) circle [radius = 0.1] node[above=1.5mm] {$\begin{pmatrix}1 & 0 \\ 3 & 4\end{pmatrix}$};
	 				\draw[, thick] (-1.732, -5) -- (-2.2, -5.7);
	 				\draw[, thick] (-1.732, -5) -- (-2.5, -5);
	 				\draw[, thick] (0, -4) -- (-1.732, -5);
	 				\draw[, thick] (0, -4) -- (1.732, -5);
	 				\filldraw  (1.732, -5) circle [radius = 0.1] node[above right = 1 mm] {$\begin{pmatrix}1 & 0 \\ 1 & 4\end{pmatrix}$};
	 				\draw[, thick] (1.732, -5) -- (2.1, -5.7);
	 				\draw[, thick] (1.732, -5) -- (2.5, -5.1);
	 				\filldraw[] (3.464, 2)  circle [radius=0.1] node[below = 1.5mm] {$\begin{pmatrix}1 & 0 \\ 0 & 2\end{pmatrix}$};
	 				\filldraw[] (-3.464, 2) circle [radius=0.1] node[below = 3mm] {$\begin{pmatrix}2 & 0 \\ 0 & 1\end{pmatrix}$};
	 				\filldraw[] (3.464, 4) circle [radius = 0.1] node[right = 1.5mm]{$\begin{pmatrix}1 & 0 \\ 2 &  4\end{pmatrix}$};
	 				\draw[, thick] (3.464, 2) -- (3.464, 4);
	 				\filldraw (5.2, 1) circle [radius = 0.1]      node[below = 1.5mm ] {$\begin{pmatrix}1 & 0 \\ 0 & 4\end{pmatrix}$};
	 				\draw[, thick] (3.464, 2) -- (5.2, 1);
	 				\draw[, thick] (3.464, 4) -- (3,4.7);
	 				\draw[, thick](3.464, 4) -- (3.9, 4.7);
	 				\draw[, thick] (5.2, 1) -- (5.9, 0.6);
	 				\draw[, thick] (5.2, 1) -- (5.9, 1.4);
	 				\filldraw[] (-3.464, 4) circle [radius = 0.1] node[right = 1.5mm]{$\begin{pmatrix}4 & 0 \\ 0 & 1\end{pmatrix}$};
	 				\draw[, thick] (-3.464, 2) -- (-3.464, 4);
	 				\filldraw (-5.2, 1) circle [radius = 0.1] node[below = 1.5mm]{$\begin{pmatrix}2 & 0 \\ 1 & 2\end{pmatrix}$};
	 				\draw[, thick] (-3.464, 2) -- (-5.2, 1);
	 				\draw[, thick] (-3.464, 4) -- (-3,4.7);
	 				\draw[, thick] (-3.464, 4) -- (-3.9, 4.7);
	 				\draw[, thick] (-5.2, 1) -- (-5.9, 0.6);
	 				\draw[, thick] (-5.2, 1) -- (-5.9, 1.4);
	 		\end{tikzpicture}}
	 		\caption{The building $\Bfrak_2$ over $K = \mathbb Q_2$.}
	 		\label{fig:buildingb2}
	 	\end{figure}
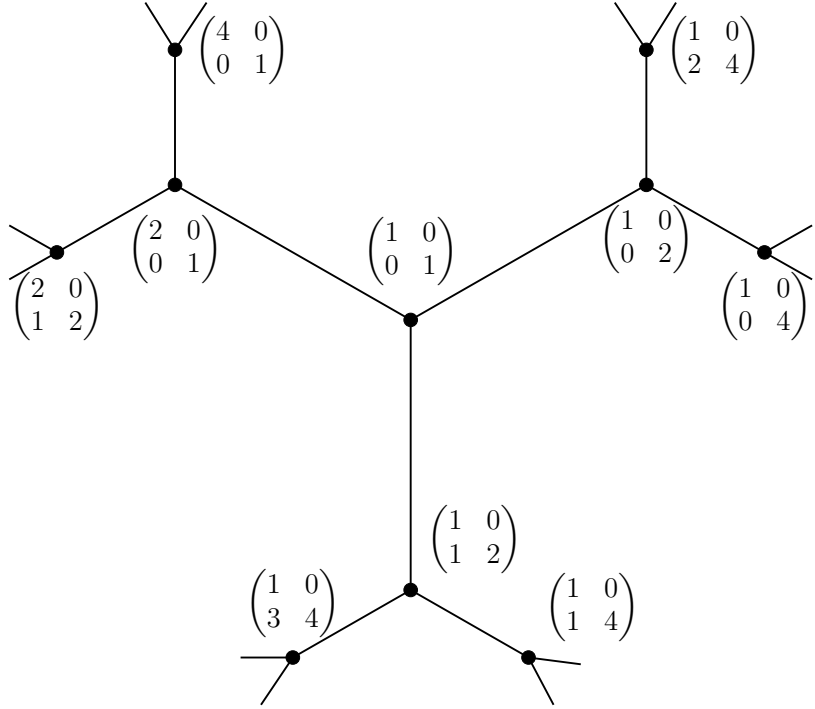
	 	
	 \end{example}

	 \subsection{Arithmetic foundations}
	 \label{subsec:arithmfound}
	 To define Mustafin varieties, we begin by defining the projective space associated to a lattice. Let $\Lambda$ be a lattice in a finite dimensional vector space $V$ over $K$. Let $V^{\vee} = \Hom_{K}(V,K)$ be the \emph{dual space} of $V$ and $\Lambda^{\vee} = \Hom_{\Ocal_K}(\Lambda, \Ocal_K)$ the \emph{dual lattice} of $\Lambda$. We denote by $\Sym(V^{\vee})$ and $\Sym(\Lambda^{\vee})$ their corresponding symmetric algebras. We define the projective spaces of $V$ and $\Lambda$ as the $\Proj$ constructions of the symmetric algebras $\Sym(V^{\vee})$ and $\Sym(\Lambda^{\vee})$, in other words,
	 \begin{align}
	 	&\mathbb{P}(V) = \mathrm{Proj}(\mathrm{Sym}(V^{\vee})), \\
	 	&\mathbb{P}(\Lambda) = \mathrm{Proj}( \Sym(\Lambda^{\vee})).
	 \end{align}
	 For readers that are more familiar with polynomial rings, we give another description of these spaces. Let us fix an $\Ocal_K$-basis $(e_1, \dots, e_d)$ of $\Lambda$ which is then also a $K$-basis of $V$. The dual basis $(e_1^\vee, \dots, e_d^\vee)$ is a basis of $V^{\vee}$ and by sending $e_i^\vee$ to the variable $x_i$, the symmetric algebra $\Sym(V^\vee)$ is isomorphic to the polynomial ring $K[x_1,\dots, x_d]$. So we can also define the projective space $\mathbb{P}(V)$ as $\Proj(K[x_1,\dots, x_d])$, i.e.\ the set of homogeneous ideals of $K[x_1, \dots, x_d]$ not containing the maximal ideal generated by $x_1, \dots, x_d$. We write
	 \[
	 \mathbb{P}(V) = \mathrm{Proj}(\mathrm{Sym}(V^\vee)) \cong \Proj(K[x_1, \dots, x_d]).
	 \]	
	 Similarly, the dual basis $(e_1^\vee, \dots, e_d^\vee)$ is a basis of $\Lambda^\vee$, and sending $e_i^\vee$ to the variable $x_i$ we can see that the symmetric algebra $\Sym(\Lambda^\vee)$ is isomorphic to the polynomial ring $\Ocal_K[x_1,\dots, x_d]$. So we have:
	 \[	
	 \mathbb{P}(\Lambda)  = \mathrm{Proj}( \mathrm{Sym}( \Lambda^\vee )) \cong \Proj(\Ocal_K[x_1,\dots, x_d]). 
	 \]
	 This is the scheme of homogeneous ideals in $\Ocal_K[x_1,\dots, x_d]$ that do not contain the ideal generated (in $\Ocal_K[x_1, \dots, x_d]$) by $x_1, \dots, x_d$. We warn the reader that, unlike the case of $K[x_1,\dots, x_d]$, this ideal is no longer maximal.\\
	 \noindent The natural embedding of rings $i : \Ocal_K \hookrightarrow \Sym(\Lambda^\vee) = \Ocal_K[x_1,\dots, x_d]$, induces a morphism of schemes
	 \[
	 	\label{equ:morph}
	 	\gamma\colon\mathbb{P}(\Lambda)  \longrightarrow \mathrm{Spec}(\mathcal{O}_K) = \{(0), 		 \mfrak_K\},\quad \pfrak \longmapsto i^{-1}(\pfrak).
	 \]
	 Hence, the scheme $\mathbb{P}(\Lambda)$ is a projective scheme over $\mathrm{Spec}(\mathcal{O}_K)=\{(0),\mathfrak{m}_K\}$. The zero ideal in $\Spec(\Ocal_K)$ is called the \textit{generic point} of $\mathrm{Spec}(\mathcal{O}_K)$, while $\mathfrak{m}_K$ is called its \textit{special point}. This is because $\mathfrak{m}_K$ lies in the closure of $\{(0)\}$ with respect to the Zariski topology. Moreover, we note that the ideal $(0)$ naturally corresponds to the field $K$, while the ideal $\mathfrak{m}_K$ naturally corresponds to the residue field $\overline{\K}$. Indeed, the natural maps $\mathcal{O}_K \to K$ and $\mathcal{O}_K\to \overline{\K}$ induce the immersions $\mathrm{Spec}(K)\to\mathrm{Spec}(\mathcal{O}_K)$, whose image is $\{(0)\}$, and $\mathrm{Spec}(\overline{\K})\to\mathrm{Spec}(\mathcal{O}_K)$, whose image is $\{\mathfrak{m}_K\}$.
	 
	 \noindent
	 The following lemma is a well-known fact that follows from base change techniques, e.g. from \cite[Prop.~3.1.9]{zbMATH05048200}.

	 \begin{lemma}
	 	The fiber over the generic point in \cref{equ:morph} is canonically isomorphic to $\mathbb{P}(V)$, while the fiber over the special point is canonically isomorphic to $\mathbb{P}_{\overline{\K}}^{d-1}$.
	 \end{lemma}
	 
	 \noindent
	 For readers unfamiliar with such notions, we give an indication as to how $\mathbb{P}(V)$ embeds into $\mathbb{P}(\Lambda)$ as its \textit{generic fiber} $\mathbb{P}(\Lambda)_K$, i.e. the preimage of $\{(0)\}$ under the morphism of schemes $\gamma$ in (\ref{equ:morph}). For this, we recall that the morphism $\gamma$ is induced by the natural injection 
	 \begin{equation}
	 	i \colon \mathcal{O}_K \hookrightarrow \mathrm{Sym}(\Lambda^\vee).
	 \end{equation}
	 Let $S_+^{\Ocal_K} \subset \Ocal_K[x_1, \dots, x_d]$ be the ideal generated by $x_1,\dots,x_d$. Then, the topological space underlying $\mathbb{P}(\Lambda)$ consists of all homogeneous prime ideals of $\mathrm{Sym}(\Lambda^\vee)$ that do not contain $S_+^{\Ocal_K}$. Similarly, let $S_+^K \subset K[x_1,\dots,x_d]$ be the ideal generated by $x_1, \dots, x_d$. Then, the scheme $\mathbb{P}(V)$ consists of all homogeneous prime ideals of $\mathrm{Sym}(V^\vee) \cong K[x_1,\dots,x_d]$ that do not contain $S_+^K$. Let $\mathfrak{p}\in\mathbb{P}(\Lambda)$. Since by definition we have $\gamma(\mathfrak{p})=i^{-1}(\mathfrak{p})$, we obtain that $\gamma(\mathfrak{p})=(0)$ if and only if we have $\pi\notin\mathfrak{p}$. Similarly, we obtain that $\gamma(\mathfrak{p})=\mathfrak{m}_K$ if and only if we have $\pi\in\mathfrak{p}$. To show that $\mathbb{P}(V)$ embeds into $\mathbb{P}(\Lambda)$, we consider the following injection
	 \[
	 j \colon \mathrm{Sym}(\Lambda^\vee) \hookrightarrow \mathrm{Sym}(V^\vee).
	 \]
	 It is enough to show that $j$ induces a bijection
	 \begin{equation}
	 	\begin{tikzcd}
	 		\phi\colon \mathbb{P}(V) = 
	 		\begin{cases}
	 			\begin{rcases}
	 				\mathfrak{p} \subseteq\mathrm{Sym}(V^\vee)\\
	 				\textrm{homogeneous}\\ \textrm{prime and}\, S_+^K\not\subseteq\mathfrak{p}
	 			\end{rcases}
	 		\end{cases}  \arrow["{\mathfrak{p}\mapsto j^{-1}(\mathfrak{p})}"]{rrr}
	 		& & &
	 		\begin{cases}
	 			\begin{rcases}
	 				\mathfrak{q} \subseteq\mathrm{Sym}(\Lambda^\vee)\\
	 				\textrm{homogeneous prime }\\ 
	 				\mathrm{and } \ \pi \not\in \qfrak ,\textrm{and}\, S_+^{\mathcal{O}_K}\not\subseteq\mathfrak{q}
	 			\end{rcases}
	 		\end{cases}
	 	\end{tikzcd}
	 \end{equation}
	 It is not so hard to see that $\phi$ is indeed a bijection by observing that 
	 \begin{equation}
	     j^{-1}(\mathfrak{p}) = j^{-1}(\mathfrak{p} \cap j(\Sym(\Lambda^\vee)))
	 \end{equation}
	 and the fact that 
	 \begin{equation}
	  \pi \in j^{-1}(\mathfrak{p} \cap j(\Sym(\Lambda^\vee))
	 \end{equation}
	 implies $\mathfrak{p} = \mathrm{Sym}(V^\vee)$.   
	 Similarly, one may see from the natural surjection
	 \begin{equation}
	 \begin{tikzcd}
	   	s\colon\mathrm{Sym}(\mathrm{Hom}_{\mathcal{O}_K}(\Lambda,\mathcal{O}_K)) \arrow[two heads]{rrr}& & & \mathrm{Sym}\left(\mathrm{Hom}_{\overline{\K}}\left(\faktor{\Lambda}{\pi \Lambda},\overline{\K}\right) \right)  
	 \end{tikzcd}
	 \end{equation}
	 that the special fiber $\mathbb{P}(\Lambda)_{\overline{\K}} = \gamma^{-1}(\{ \mfrak_K \})$ of $\mathbb{P}(\Lambda)$ is isomorphic to the projective space $\mathbb{P}\left(\faktor{\Lambda}{\pi \Lambda}\right)\cong\mathbb{P}_{\overline{\K}}^{d-1}$ over $\mathrm{Spec}(\overline{\K})$. We summarize our discussion in \cref{Fig:projLattice}.
	 
	 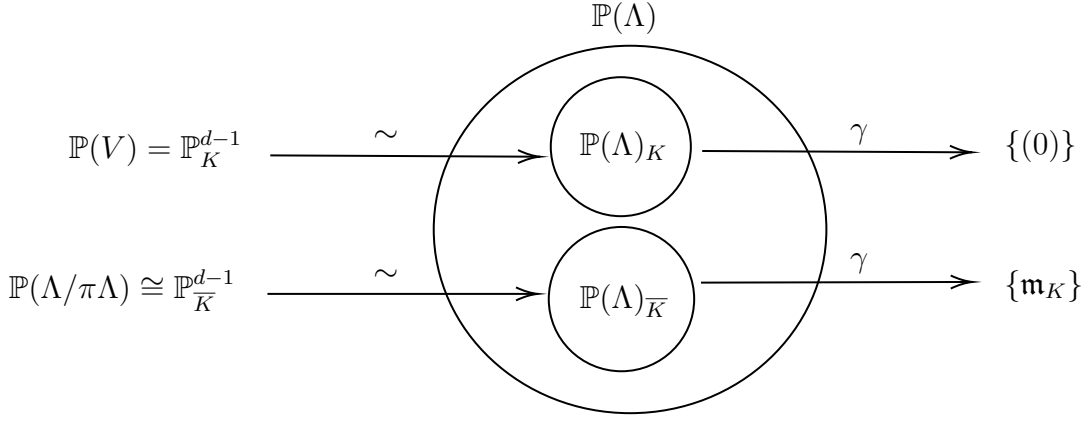
\begin{figure}
	 
	 	\begin{center}
	 		\tikzset{every picture/.style={line width=0.75pt}}
	 		
	 		\begin{tikzpicture}[x=0.75pt,y=0.75pt,yscale=-1,xscale=1]
	 			
	 			\draw   (243.58,241.26) .. controls (243.58,190.22) and (287.46,148.85) .. (341.58,148.85) .. controls (395.71,148.85) and (439.58,190.22) .. (439.58,241.26) .. controls (439.58,292.29) and (395.71,333.67) .. (341.58,333.67) .. controls (287.46,333.67) and (243.58,292.29) .. (243.58,241.26) -- cycle ;
	 			
	 			\draw   (302,199.62) .. controls (302,180.32) and (317.65,164.67) .. (336.96,164.67) .. controls (356.27,164.67) and (371.92,180.32) .. (371.92,199.62) .. controls (371.92,218.93) and (356.27,234.58) .. (336.96,234.58) .. controls (317.65,234.58) and (302,218.93) .. (302,199.62) -- cycle ;
	 			
	 			\draw   (301,276.29) .. controls (301,256.25) and (317.25,240) .. (337.29,240) .. controls (357.33,240) and (373.58,256.25) .. (373.58,276.29) .. controls (373.58,296.33) and (357.33,312.58) .. (337.29,312.58) .. controls (317.25,312.58) and (301,296.33) .. (301,276.29) -- cycle ;
	 			
	 			\draw    (377.58,201.79) -- (510.17,202.45) ;
	 			\draw [shift={(512.17,202.46)}, rotate = 180.28] [color={rgb, 255:red, 0; green, 0; blue, 0 }  ][line width=0.75]    (10.93,-3.29) .. controls (6.95,-1.4) and (3.31,-0.3) .. (0,0) .. controls (3.31,0.3) and (6.95,1.4) .. (10.93,3.29)   ;
	 			
	 			\draw    (377,268) -- (396.58,267.95) -- (510.58,267.67) ;
	 			\draw [shift={(512.58,267.67)}, rotate = 539.86] [color={rgb, 255:red, 0; green, 0; blue, 0 }  ][line width=0.75]    (10.93,-3.29) .. controls (6.95,-1.4) and (3.31,-0.3) .. (0,0) .. controls (3.31,0.3) and (6.95,1.4) .. (10.93,3.29)   ;
	 			
	 			\draw    (162,204) -- (294.58,204.66) ;
	 			\draw [shift={(296.58,204.67)}, rotate = 180.28] [color={rgb, 255:red, 0; green, 0; blue, 0 }  ][line width=0.75]    (10.93,-3.29) .. controls (6.95,-1.4) and (3.31,-0.3) .. (0,0) .. controls (3.31,0.3) and (6.95,1.4) .. (10.93,3.29)   ;
	 			
	 			\draw    (160.58,273.67) -- (293.58,273.67) ;
	 			\draw [shift={(295.58,273.67)}, rotate = 180] [color={rgb, 255:red, 0; green, 0; blue, 0 }  ][line width=0.75]    (10.93,-3.29) .. controls (6.95,-1.4) and (3.31,-0.3) .. (0,0) .. controls (3.31,0.3) and (6.95,1.4) .. (10.93,3.29)   ;

	 			\draw (315,190.4) node [anchor=north west][inner sep=0.75pt]  [xscale=1,yscale=1]  {$\mathbb{P}( \Lambda )_{K}$};
	 			
	 			\draw (315,267.4) node [anchor=north west][inner sep=0.75pt]  [xscale=1,yscale=1]  {$\mathbb{P}( \Lambda )_{\overline{\K}}$};
	 			
	 			\draw (321,125.4) node [anchor=north west][inner sep=0.75pt]  [xscale=1,yscale=1]  {$\mathbb{P}( \Lambda )$};
	 			
	 			\draw (527,190) node [anchor=north west][inner sep=0.75pt]  [xscale=1,yscale=1]  {$\{(0)\}$};
	 			
	 			\draw (527,260) node [anchor=north west][inner sep=0.75pt]  [xscale=1,yscale=1]  {$\{\mathfrak{m}_K\}$};
	 			
	 			\draw (30,260) node [anchor=north west][inner sep=0.75pt]  [xscale=1,yscale=1]  {$\mathbb{P}(\Lambda / \pi \Lambda) \cong \mathbb{P}_{\overline{\K}}^{d-1}$};
	 			
	 			\draw (60,190) node [anchor=north west][inner sep=0.75pt]  [xscale=1,yscale=1]  {$\mathbb{P}(V)$};
	 			
	 			\draw (212,190) node [anchor=north west][inner sep=0.75pt]  [xscale=1,yscale=1]  {$\sim $};
	 			
	 			\draw (212,260) node [anchor=north west][inner sep=0.75pt]  [xscale=1,yscale=1]  {$\sim $};
	 			
	 			\draw (450,187) node [anchor=north west][inner sep=0.75pt]  [xscale=1,yscale=1]  {$\gamma$};
	 			
	 			\draw (450,252) node [anchor=north west][inner sep=0.75pt]  [xscale=1,yscale=1]  {$\gamma$};

	 		\end{tikzpicture}
	 		\caption{Components of $\mathbb{P}(\Lambda)$}
	 		\label{Fig:projLattice}
	 	\end{center}

	 \end{figure}

	 \subsection{Images of rational maps} \label{imagesOfRationalMaps}
	 \label{pre-rat}
	 In \cite{li2018images}, a combinatorial framework for the study of images of rational maps was developed. More precisely, let $F$ be an infinite field\footnote{The results in \cite{li2018images} are stated for algebraically closed fields, but hold for any infinite field $F$.}, and fix two positive integers $n,d$ and non-zero matrices $g_1,\dots,g_n\in\mathrm{Mat}_{d}(F)$. We define the rational map
	 \begin{equation}
	 	\begin{tikzcd}
	 		\underline{g}:\mathbb{P}_F^{d-1}\arrow[dashed,"{(g_1,\dots,g_n)}"]{rrr}& & & \left(\mathbb{P}_F^{d-1}\right)^{n}
	 	\end{tikzcd}
	 \end{equation}
	 by sending a column vector $x = (x_1: \dots : x_d)^\top \in \mathbb{P}_F^{d-1}$ to $\underline{g}(x) = (g_1 x , \dots, g_n x ) \in (\mathbb{P}_F^{d-1})^n$. This map is well defined on a Zariski dense set in $\mathbb{P}_F^{d-1}$) (hence the name rational map). We then define $X(\underline{g})$ to be the Zariski closure of the image of $\underline{g}$, i.e.
	 \begin{equation}
	 	X(\underline{g}) = \overline{\mathrm{Im}(\underline{g})}.
	 \end{equation}
	 For each $i=1, \dots, n$, we denote $V_i=\mathrm{ker}(g_i)$ and for $I \subseteq \{1,\dots,n\}$ we define
	 \[
	 d_I=\mathrm{dim}\bigcap_{i\in I} V_i.
	 \]
	 For any integer $h \geq 0$, set
	 
	 \begin{equation}
	 	M(h)=\left\{\underline{m}\in\mathbb{N}^n \colon \sum_{i = 1}^{n} m_i=h\,\textrm{and}\, \textrm{for all}\,I\subseteq\{1,\dots,n\}, \sum_{i\in I} m_i<d-d_I\ \right\}.
	 \end{equation}
	 The following result on the varieties $X(\underline{g})$ was derived in \cite{li2018images}.
	 
	 \begin{theorem}[\cite{li2018images}]
	 	Let $n,d,g_1,\dots,g_n,M(h)$ as above and set
	 	\begin{equation}
	 	d_{\underline{g},n}=\mathrm{max}\{h \colon M(h)\neq\emptyset\}. 	 
	 	\end{equation}
    Then, we have $\mathrm{dim}(X(\underline{g}))=d_{\underline{g},n}$.
	 \end{theorem}
	 
	 \subsection{Mustafin varieties}
	 \label{sec:mustafin}

	 	Let $\Gamma=\{[\Lambda_1],\dots,[\Lambda_n]\}$ be a set of homothety classes of lattices. We consider the open immersions
	 	\begin{equation}
	 		f_i\colon\mathbb{P}(V)\to\mathbb{P}(\Lambda_i), \ i=1,\dots,n
	 	\end{equation}
	 	mapping to the generic fiber and define the map
	 	\begin{equation}
	 		\underline{f}\colon\mathbb{P}(V) \xrightarrow[]{(f_1,\dots,f_n)} \mathbb{P}(\Lambda_1)\times\dots \times \mathbb{P}(\Lambda_n).
	 	\end{equation}
	 	
	 	\begin{definition}
	 	The \emph{Mustafin variety} $\Mcal(\Gamma)$ associated to $\Gamma$ is the Zariski closure of the image of $\underline{f}=(f_1,\dots,f_n)$ endowed with the reduced scheme structure.
	 	\end{definition}
	 	
	 \begin{theorem}[\cite{MR2861606}]
	 	The Mustafin variety $\mathcal{M}(\Gamma)$ is a flat, projective and proper scheme over $\mathrm{Spec}(\mathcal{O}_K)$. Also, the generic fiber $\mathcal{M}(\Gamma)_K$ is isomorphic to $\mathbb{P}(V)$. The special fiber $\mathcal{M}(\Gamma)_{\overline{\K}}$ is reduced and hence is a variety.
	 \end{theorem}
	 
	 \noindent
	 The special fiber of Mustafin varieties may be studied in terms of the theory of convex hulls in Bruhat-Tits buildings.
	 
	 \begin{definition}
	 	The \textit{reduction complex} of $\mathcal{M}(\Gamma)$ is the simplicial complex where the vertices are the irreducible components of $\mathcal{M}(\Gamma)_{\overline{\K}}$ and where a set of vertices forms a simplex if and only if the intersection of the respective irreducible components is non-empty.
	 \end{definition}
	 
	 \begin{theorem}[\cite{MR2861606}]
	 	Let $\Gamma$ be a finite collection of lattice classes and assume that $\Gamma=\mathrm{conv}(\Gamma)$. Then, the reduction complex of $\mathcal{M}(\Gamma)$ is isomorphic to the simplicial complex $\mathrm{conv}(\Gamma)$.
	 \end{theorem}
	 
	 \noindent
	 We now discuss how to compute the special fiber $\mathcal{M}(\Gamma)_{\overline{\K}}$. Let $\Gamma=\{[\Lambda_1],\dots,[\Lambda_n]\}$ be a finite set of homothety classes of lattices. We fix a \textit{reference lattice} $\Lambda$ and matrices $g_1,\dots,g_n\in\GL(V)$, such that $g_i\Lambda=\Lambda_i$.  Then, the following diagram is commutative 
	 
	 \begin{equation}
	 	\label{equ-comdiagram}
	 	\begin{tikzcd}
	 		\mathbb{P}(V) \arrow["{(g_1^{-1},\dots,g_n^{-1})}"]{rr}\arrow["{(f_1,\dots,f_n)}"]{dd}& & \mathbb{P}(V)^n\arrow["{(f,\dots,f)}"]{dd}\\
	 		\\
	 		\mathbb{P}(\Lambda_1)\times\dots\times\mathbb{P}(\Lambda_n) \arrow["{(g_1^{-1},\dots,g_n^{-1})}"]{rr}& & \mathbb{P}(\Lambda)^n
	 	\end{tikzcd}
	 \end{equation}
	 We denote the coordinates of the $j$-th factor in $\mathbb{P}(\Lambda)^n$ by $x_{1j},\dots,x_{dj}$, and define the homogeneous ideal $I_{\underline{g}}$ in $K\left[\left(x_{ij}\right)_{\substack{i=1,\dots,d\\j=1,\dots,n}}\right]$ as the ideal generated by the $2\times 2$ minors of the matrix
	 
	 \[
	\left( \begin{array}{c|c|c}
	 	g_1\begin{pmatrix}
	 		x_{11}\\
	 		\vdots\\
	 		x_{d1}
	 	\end{pmatrix} 
	 	&\hspace*{10pt}\hdots\hspace*{10pt} &g_n\begin{pmatrix}
	 		x_{1n}\\
	 		\vdots\\
	 		x_{dn}
	 	\end{pmatrix}
	 \end{array}\right).
	 \]
	 Then, the Mustafin variety $\mathcal{M}(\Gamma)$ is isomorphic to the subscheme in $\mathbb{P}(\Lambda)^n$ cut out by the ideal $I(\Gamma) =   I_{\underline{g}} \cap \Ocal_K\left[\left(x_{ij}\right)_{\substack{i=1,\dots,d\\j=1,\dots,n}}\right]$. Moreover, the special fiber is isomorphic to the variety in $(\mathbb{P}_{\overline{\K}}^{d-1})^n$ cut out by the ideal $I(\Gamma)_{\overline{\K}}$ in $\overline{\K}\left[\left(x_{ij}\right)_{\substack{i=1,\dots,d\\j=1,\dots,n}}\right]$ given by
	 \[
	 I(\Gamma)_{\overline{\K}} = I(\Gamma) \mod \mfrak_K.
	 \]
	 
	 \begin{example}
	 	Let $K=\mathbb{Q}_2$ (then $\Ocal_K = \Z_2$) and fix three matrices
	 	\begin{equation}
	 		g_1 = \begin{pmatrix}
	 			3008 & 1088 & 304\\
	 			432 & 40 & 416 \\
	 			36 & 344 & 100
	 		\end{pmatrix},\quad
	 		g_2 = \begin{pmatrix}
	 			94 & 5376 & 3328\\
	 			6 & 1792 & 192\\
	 			48 & 160 & 196
	 		\end{pmatrix},\quad
	 		g_3 = \begin{pmatrix}
	 			3 & 592 & 16\\
	 			376 & 18 & 656\\
	 			256 & 40 & 3072
	 		\end{pmatrix}.
	 	\end{equation}
	 	Let $\Lambda_1=g_1 \Z_{2}^3,\Lambda_2 = g_2 \Z_{2}^3,\Lambda_3 = g_3 \Z_{2}^3$ and $\Gamma=\{[\Lambda_1],[\Lambda_2],[\Lambda_3]\}$. Then the Mustafin variety $\mathcal{M}(\Gamma)$ is the subscheme of $\mathbb{P}( \Z_2^3)$ cut out by 
	 	\begin{align}
	 		I(\Gamma) = I_{(g_1,g_2,g_3)} \cap \Z_2 \left[\left(x_{ij}\right)_{\substack{i=1,\dots,3\\j=1,\dots,3}}\right],
	 	\end{align}
	 	where $I_{(g_1,g_2,g_3)}$ is the ideal in $\Q_2\left[\left(x_{ij}\right)_{\substack{i=1,\dots,3\\j=1,\dots,3}}\right]$ generated by the $2 \times 2$ minors of the matrix
	 	\begin{equation}
	 	\begin{pmatrix}
	 		3008x_{11}+1088x_{21}+304x_{31} & 94x_{12}+5376x_{22}+3328x_{32} & 3x_{13}+592x_{23}+16x_{33}\\
	 		432x_{11}+40x_{21}+416x_{31} & 6x_{12}+1792x_{22}+192x_{32} & 376x_{13}+18x_{23}+656x_{33}\\
	 		36x_{11}+344x_{21}+100x_{31} & 48x_{12}+160x_{22}+196x_{32} & 256x_{13}+40x_{23}+3072x_{33}
	 	\end{pmatrix}
	 	\end{equation}
	 This can be computed by saturating the ideal generated over $\mathcal{O}_K$ by the $2\times 2$-minors with respect to the prime $(2)$. A generating set of this saturated ideal will still be a generating set when modding out the maximal ideal $\mathfrak{m}_K=(2)_{\mathcal{O}_K}$. We can compute this using \textsc{Singular} \cite{DGPS}, and we get the following ideal in $\F_2[(x_{ij})]$:
	 \begin{align}
	 	I(\Gamma)_{\F_2} = &\langle x_{32}x_{23},x_{11}x_{23}+x_{31}x_{23},x_{32}x_{13},x_{12}x_{13},x_{21}x_{13},x_{11}x_{13}+x_{21}x_{13}+x_{31}x_{13},\\
	 	&x_{21}x_{32},x_{21}x_{12},x_{11}x_{12}+x_{31}x_{12}\rangle\\
	 	=&\langle x_{23},x_{13},x_{32},x_{12}\rangle \cap\langle x_{23},x_{13},x_{12},x_{21}\rangle \cap \langle x_{23},x_{13},x_{21},x_{11}+x_{31}\rangle \cap\\
	 	&\langle x_{13},x_{32},x_{12},x_{11}+x_{31}\rangle\cap\langle x_{13},x_{32},x_{21},x_{11}+x_{31}\rangle\cap\langle x_{32},x_{12},x_{21},x_{11}+x_{31}\rangle .
	 \end{align}
	 The first, third and sixth ideal cut out a variety isomorphic to $\mathbb{P}_{\overline{\K}}^2$ and the second, fourth and fifth ideal cut out a variety isomorphic to $
	 \mathbb{P}_{\overline{\K}}^1\times\mathbb{P}_{\overline{\K}}^1$. Thus, the special fiber of $\mathcal{M}(\Gamma)$ is union of six irreducible components: three copies of $\mathbb{P}_{\overline{\K}}^2$ and three copies of $\mathbb{P}_{\overline{\K}}^1\times\mathbb{P}_{\overline{\K}}^1$.
	 \end{example}
	 
	 \subsection{Classification of special fibers} 
    \label{subsec:class}	 
	 In this section, we present a few results on the classification of irreducible components of Mustafin varieties. As we shall see, these irreducible components will be of the form $X(\underline{g})$ from Section \ref{imagesOfRationalMaps}.
	 
	 Let $\Gamma = \{ [\Lambda_1], \dots, [\Lambda_n] \}$ be a finite set of vertices in the Bruhat-Tits building $\Bfrak_d$. We fix another point $[\Lambda] \in \Bfrak_d^0$ with a representative lattice $\Lambda$. Then, there exist $g_1, \dots, g_n \in \GL(V)$ such that $\Lambda_i = g_i \Lambda$. As explained in Subsection \ref{sec:mustafin}, we get a map
	 
	 \begin{equation}
	 	\begin{tikzcd}
	 		\underline{g}\colon\mathbb{P}(V) \arrow["{(g_1^{-1},\dots,g_n^{-1})}"]{rrr}& & & \mathbb{P}(\Lambda)^n
	 	\end{tikzcd}.
	 \end{equation}
	 This map $\underline{g}$ naturally extends to a continuous rational map on $\mathbb{P}(\Lambda)$, that is
	 \begin{equation}
	 	\begin{tikzcd}
	 		\underline{g}\colon\mathbb{P}(\Lambda) \arrow[dashed,"{(g_1^{-1},\dots,g_n^{-1})}"]{rrr}& & & \mathbb{P}(\Lambda)^n.
	 	\end{tikzcd}
	 \end{equation}
	 We then get a map $\underline{g}_{\overline{\K}}$ on the special fiber of $\mathbb{P}(\Lambda)$, induced by $\underline{g}$:
	 \begin{equation}
	 	\label{equ-mapspecfib}
	 	\begin{tikzcd}
	 		\underline{g}_{\overline{\K}}\colon\mathbb{P}_{\overline{\K}}^{d-1}\cong\mathbb{P}(\Lambda)_{\overline{\K}} \arrow[dashed,"{( g_1^{-1},\dots,g_n^{-1})_{\overline{\K}}}"]{rrr}& & & \mathbb{P}(\Lambda)_{\overline{\K}}^n\cong\left(\mathbb{P}_{\overline{\K}}^{d-1}\right)^n.
	 	\end{tikzcd}
	 \end{equation}
	 We now give an explicit description of this map $\underline{g}_{\overline{\K}}$. For a fixed $\Ocal_K$-basis of $\Lambda$, let $A_i$ be the matrix representing $g_i$ in this basis. Let $\tilde{A_i}$ be the saturation of $A_i$ with respect to the prime ideal $(\pi) = \pi \Ocal_K$, i.e. if $s$ is the minimal valuation of the entries in $A_i$, then $\tilde{A_i}=\pi^{-s}A_i$. The map $\underline{g}$ is induced  by the matrices $A_1, \dots, A_n$ via the map
	 
	 \[
	 \Lambda  \xrightarrow{ ( A_1, \dots ,  A_n)}   \Lambda^n,\quad  x  \longmapsto (A_1 x, \dots, A_n x).
	 \]
	 The map $\underline{g}_{\overline{\K}}$ on the special fiber is then induced by the matrices $\overline{A_i} \equiv \tilde{A}_i \mod \pi$ as follows:
	 
	 \begin{equation}
	 	\begin{tikzcd}
	 		\underline{g}_{\overline{\K}} \colon\mathbb{P}_{\overline{\K}}^{d-1}\cong\mathbb{P}(\Lambda)_{\overline{\K}} \arrow[dashed,"{(\overline{A_1}, \dots, \overline{A}_n )}"]{rrrrr}& & & & & \mathbb{P}(\Lambda)_{\overline{\K}}^n\cong\left(\mathbb{P}_{\overline{\K}}^{d-1}\right)^n.
	 	\end{tikzcd}
	 \end{equation}
	 We can then associate to the point $[\Lambda] \in \Bfrak_d^0$ the variety $X_{[\Lambda], \Gamma}$ defined as the closure of the image of $\underline{g}_{\overline{\K}}$, i.e,
	 \[
	 X_{[\Lambda],\Gamma}= \overline{\mathrm{Im}(\underline{g}_{\overline{\K}})} \subseteq (\mathbb{P}(\Lambda)_{\overline{\K}})^n.
	 \]
	 Notice that the variety $X_{[\Lambda], \Gamma}$ lives in $(\mathbb{P}(\Lambda)_{\overline{\K}})^n$. We embed every such variety in the common space $(\mathbb{P}(\Ocal_K^d)_{\overline{\K}} ) ^n$. To do so, let $g_{\Lambda} \in \GL(V)$ be such that $\Lambda = g_{\Lambda} \Ocal_K^d $ and consider the isomorphism of schemes
	 \begin{equation}
	 	\begin{tikzcd}
	 		\underline{g}_{\Lambda}\colon\mathbb{P}(\Lambda)^n \arrow["{(g_{\Lambda }^{-1},\dots,g_{\Lambda}^{-1})}"]{rrr}& & & \mathbb{P}(\mathcal{O}_K^d)^n,
	 	\end{tikzcd}
	 \end{equation}
	 that maps the projective space of the lattice $\Lambda$ into the projective space of the standard lattice $\Ocal_K^d$. This isomorphism naturally enduces an ismorphism $\left(\underline{g}_{\Lambda}\right)_{\overline{\K}}$ of special fibers
	 \[
	 \left(\underline{g}_{ \Lambda }\right)_{\overline{\K}}: (\mathbb{P}(\Lambda)_{\overline{\K}})^n \longrightarrow (\mathbb{P}(\mathcal{O}_K^d)_{\overline{\K}})^n.
	 \]
	 The map $ \left(\underline{g}_{ \Lambda } \right)_{\overline{\K}}$ then maps the variety $X_{[\Lambda], \Gamma}$ into $(\mathbb{P}(\mathcal{O}_K^d)_{\overline{\K}})^n$, and we define
	 
	 \begin{equation}
	 	\mathcal{N}(\Gamma)=\bigcup_{[\Lambda] \in \Bfrak_d^0} \left(\underline{g}_{ \Lambda } \right)_{\overline{\K}}(X_{[\Lambda],\Gamma}) \subseteq\left(\mathbb{P}_K^{d-1}\right)^n= (\mathbb{P}(\mathcal{O}_K^d)_{\overline{\K}})^n,
	 \end{equation}
	 which we endow with the reduced scheme structure. The following is a reformulation of the main results in \cite{hahn2020mustafin}.\footnote{This classification result in \cite{hahn2020mustafin} was originally stated for $\overline{K}$ algebraically closed. However, the result is valid for any discretely valued field with infinite residue field with the same proof.}
	 
	 \begin{theorem}[{\cite[Thm.~1.2, Lem.~3.7, and Rmk.~3.14]{hahn2020mustafin}}]
	 	\label{thm-hlmust}
	 	We assume that $\overline{K}$ is infinite. Let $\Gamma=\{[\Lambda_1],\dots,[\Lambda_n]\}$ be a finite set of homothety classes of lattices and consider the embedding $\mathcal{M}(\Gamma)\subseteq\mathbb{P}(\mathcal{O}_K^d)^n$ induced by $\mathbb{P}(V)\xrightarrow{(g_{\Lambda_1},\dots,g_{ \Lambda_n })}\mathbb{P}(\mathcal{O}_K^d)^n$. For any $[\Lambda]\in\mathfrak{B}_d^0$, the following hold:
	 	\begin{enumerate}[label=$(\arabic*)$ ]
	 		\item  $\left(\underline{g}_{\Lambda } \right)_{\overline{K}} (X_{[\Lambda],\Gamma})\subseteq\mathcal{M}(\Gamma)_{\overline{\K}}\subseteq\mathbb{P}(\mathcal{O}_K)^n_{\overline{\K}}$ with respect to the above embedding,
	 		\vspace{2mm}
	 		\item $\left(\underline{g}_{ \Lambda }\right)_{\overline{K}} (X_{[\Lambda],\Gamma})$ is an irreducible component of the Mustafin variety $\mathcal{M}(\Gamma)_{\overline{\K}}$ if and only if $\mathrm{dim}(X_{[\Lambda],\Gamma})=d-1$,
	 		\vspace{2mm}
	 		\item if $\mathrm{dim}(X_{[\Lambda],\Gamma})=d-1$, then $[\Lambda]\in\mathrm{conv}(\Gamma)$.
	 	\end{enumerate}
	 	Furthermore, any irreducible component of $\mathcal{M}(\Gamma)_{\overline{\K}}$ is of the form $\left(\underline{g}_{\Lambda}\right)_{\overline{K}}(X_{[\Lambda],\Gamma})$ for a unique homothety class $[\Lambda]\in\mathrm{conv}(\Gamma)$.
	 	Finally, we have
	 	\[
	 	\mathcal{M}(\Gamma)_{\overline{\K}}=\mathcal{N}(\Gamma)=\bigcup_{[\Lambda] \in \mathrm{conv}(\Gamma)} \left(\underline{g}_{ \Lambda } \right)_{\overline{\K}}(X_{[\Lambda],\Gamma}) \subseteq\left(\mathbb{P}_K^{d-1}\right)^n= (\mathbb{P}(\mathcal{O}_K^d)_{\overline{\K}})^n.
	 	\]
	 \end{theorem}
	 
	 \noindent
	 We end this subsection with an example illustrating our construction.
	 
	 \begin{example}
	 	Let $K=\mathbb{Q}((\pi))$ and fix two lattices $\Lambda_1,\Lambda_2$ given by
	 	\begin{equation}
	 		\Lambda_1=\begin{pmatrix}
	 			1 & 0 & 0\\
	 			0 & 1 & 0\\
	 			0 & 0 & 1
	 		\end{pmatrix}\mathcal{O}_K^3\quad\textrm{and}\quad \Lambda_2= \begin{pmatrix}
	 			1 & 0 & 0\\
	 			0 & \pi & 0\\
	 			0 & 0 & \pi^2
	 		\end{pmatrix}\mathcal{O}_K^3.
	 	\end{equation}
	 	We now consider the Mustafin variety $\mathcal{M}(\Gamma)$ for $\Gamma=\{[\Lambda_1],[\Lambda_2]\}$. Using the procedure described in Section \ref{sec:mustafin}, the Mustafin variety $\Mcal(\Gamma)$ is cut by the ideal generated by the $2\times 2$ minors of the matrix
	 	\begin{equation}
	 		\begin{pmatrix}
	 			x_{11} & x_{12}\\
	 			x_{21} & \pi x_{22}\\
	 			x_{31} & \pi^2x_{32}
	 		\end{pmatrix}.
	 	\end{equation}
	 	So the special fiber $\mathcal{M}(\Gamma)_{\overline{\K}}\subseteq\mathbb{P}_{\overline{\K}}^2\times\mathbb{P}_{\overline{\K}}^2$ is cut out by
	 	\begin{equation}
	 		\langle x_{12},x_{22}\rangle_{{\overline{\K}}[x_{ij}]}\cap\langle x_{31},x_{12}\rangle_{{\overline{\K}}[x_{ij}]}\cap\langle x_{21},x_{31}\rangle_{{\overline{\K}}[x_{ij}]}.
	 	\end{equation}
	 	
	 	\noindent
	 	In other words, $\mathcal{M}(\Gamma)_{\overline{\K}}$ is the union of the three irreducible components 
	 	
	 	\[
	 	\mathbb{P}_{\overline{\K}}^2\times\{ (0:0:1) \}, \quad  \{ (1:0:0)\} \times \mathbb{P}_{\overline{\K}}^2 \quad \text{ and } \quad \mathbb{P}_{\overline{\K}}^1\times\mathbb{P}_{\overline{\K}}^1.
	 	\]
	 	 Moreover, is easily seen that $\mathrm{conv}(\Gamma)=\{[\Lambda_1],[\Lambda_2],[\Lambda_3]\}$ where
	 	\[
	 	\Lambda_3 =\begin{pmatrix}
	 		1 & 0 & 0\\
	 		0 & 1 & 0\\
	 		0 & 0 & \pi
	 	\end{pmatrix}\mathcal{O}_K^3.
	 	\]
	 	Indeed, we will now show that the first component $\mathbb{P}_{\overline{\K}}^2\times\{ (0:0:1)  \}$ corresponds to $[\Lambda_1]$, the second component $\{ (1:0:0) \}\times\mathbb{P}_{\overline{\K}}^2$ corresponds to $[\Lambda_2]$ and the last component $\mathbb{P}_{\overline{\K}}^1\times\mathbb{P}_{\overline{\K}}^1$ corresponds to the third point $[\Lambda_3]$ in the convex hull of $\Gamma$. To see that, we fix representative matrices of $\Lambda_1, \Lambda_2$ and $\Lambda_3$:
	 	\begin{equation}
	 		h_{\Lambda_1}=\begin{pmatrix}
	 			1 & 0 & 0\\
	 			0 & 1 & 0\\
	 			0 & 0 & 1
	 		\end{pmatrix},\quad h_{\Lambda_2}=\begin{pmatrix}
	 			1 & 0 & 0\\
	 			0 & \pi & 0\\
	 			0 & 0 & \pi^2
	 		\end{pmatrix},\quad h_{\Lambda_3}=\begin{pmatrix}
	 			1 & 0 & 0\\
	 			0 & 1 & 0\\
	 			0 & 0 & \pi
	 		\end{pmatrix}.
	 	\end{equation}
	 	
	 	\begin{enumerate}[label=(\arabic*)]
	 		\item For $\Lambda=\Lambda_1$ in the above procedure, we obtain
	 		\begin{equation}
	 			g_1=h_{ \Lambda_1 }h_{ \Lambda_1 }^{-1}=\begin{pmatrix}
	 				1 & 0 & 0\\
	 				0 & 1 & 0\\
	 				0 & 0 & 1
	 			\end{pmatrix}\quad\textrm{and}\quad g_2=h_{ \Lambda_2 }h_{ \Lambda_1 }^{-1}=\begin{pmatrix}
	 				1 & 0 & 0\\
	 				0 & \pi & 0\\
	 				0 & 0 & \pi^2
	 			\end{pmatrix}.
	 		\end{equation}
	 		Thus, we obtain
	 		\begin{align}
	 			\tilde{A}_1 = \begin{pmatrix}
	 				1 & 0 & 0\\
	 				0 & 1 & 0\\
	 				0 & 0 & 1
	 			\end{pmatrix}\equiv\begin{pmatrix}
	 				1 & 0 & 0\\
	 				0 & 1 & 0\\
	 				0 & 0 & 1
	 			\end{pmatrix}\,\mathrm{mod}\,\pi\textup{ and }
	 			\tilde{A}_2=\begin{pmatrix}
	 				\pi^2 & 0 & 0\\
	 				0 & \pi & 0\\
	 				0 & 0 & 1
	 			\end{pmatrix}\equiv\begin{pmatrix}
	 				0 & 0 & 0\\
	 				0 & 0 & 0\\
	 				0 & 0 & 1
	 			\end{pmatrix}\,\mathrm{mod}\,\pi
	 		\end{align}
	 		We see immediately that $X_{[\Lambda_1], \Gamma } = \overline{\mathrm{Im}(\underline{g}_{\Lambda_1})}$ is cut out by $\langle x_{12},x_{22}\rangle_{{\overline{\K}}[x_{ij}]}$.
	 		
	 		\item Similarly, for $\Lambda=\Lambda_2$ we obtain
	 		\begin{align}
	 			\tilde{A}_1 = \begin{pmatrix}
	 				1 & 0 & 0\\
	 				0 & \pi & 0\\
	 				0 & 0 & \pi^2
	 			\end{pmatrix}\equiv\begin{pmatrix}
	 				1 & 0 & 0\\
	 				0 & 0 & 0\\
	 				0 & 0 & 0
	 			\end{pmatrix}\,\mathrm{mod}\,\pi\textup{ and }
	 			\tilde{A_2}=\begin{pmatrix}
	 				1 & 0 & 0\\
	 				0 & 1 & 0\\
	 				0 & 0 & 1
	 			\end{pmatrix}\equiv\begin{pmatrix}
	 				1 & 0 & 0\\
	 				0 & 1 & 0\\
	 				0 & 0 & 1
	 			\end{pmatrix}\,\mathrm{mod}\,\pi
	 		\end{align}
	 		and thus $X_{[\Lambda_2]}=\overline{\mathrm{Im}(\underline{g}_{ \Lambda_2 })}$ is cut out by $\langle x_{21},x_{31}\rangle_{{\overline{\K}}[x_{ij}]}$.
	 		
	 		\item Finally, for $\Lambda= \Lambda_3 $ we obtain
	 		\begin{align}
	 			\tilde{A}_1=\begin{pmatrix}
	 				1 & 0 & 0\\
	 				0 & 1 & 0\\
	 				0 & 0 & \pi
	 			\end{pmatrix}\equiv\begin{pmatrix}
	 				1 & 0 & 0\\
	 				0 & 1 & 0\\
	 				0 & 0 & 0
	 			\end{pmatrix}\,\mathrm{mod}\,\pi\textup{ and }
	 			\tilde{A}_2=\begin{pmatrix}
	 				\pi & 0 & 0\\
	 				0 & 1 & 0\\
	 				0 & 0 & 1
	 			\end{pmatrix} \equiv \begin{pmatrix}
	 				0 & 0 & 0\\
	 				0 & 1 & 0\\
	 				0 & 0 & 1
	 			\end{pmatrix}\,\mathrm{mod}\,\pi
	 		\end{align}
	 		\noindent and thus $X_{[\Lambda_3]}=\overline{\mathrm{Im}(\underline{g}_{\Lambda_3})}$ is cut out by $\langle x_{31},x_{12}\rangle_{{\overline{\K}}[x_{ij}]}$.

	 	\end{enumerate}
	 \end{example}
	 
	 \subsection{Rank-metric codes and Mustafin varieties} \label{subsec:rankmet} 
	 Let $K$ be field with a non-trivial discrete valuation and let $\Lcal\subseteq\mathrm{Mat}_{d\times e}(K)$ be a linear space of matrices of dimension $n$ given by
	 \begin{equation}
	 	\mathcal{L}=\langle A_1,\dots,A_n\rangle_{K}\quad\textrm{for}\quad A_1,\dots,A_n\in\mathrm{Mat}_{d\times e}(K).
	 \end{equation}
	 
	 \noindent
	 We define $\Kcal = \Lcal \cap \Mat_{d \times e}(\Ocal_K)$. We note that $\Kcal$ is a free $\Ocal_K$-module since $\Ocal_K$ is a principal ideal domain and a submodule of the free $\Ocal_K$-module $\Mat_{d\times e}(K)$.
	 The matrices $A_1, \dots, A_n$ need not be an $\Ocal_K$-basis of $\Kcal$, even if $A_1, \dots, A_n \in \mathrm{Mat}_{d \times e}(\Ocal_K)$. A basis can however be obtained by computing the saturation of $\langle A_1,\dots,A_n\rangle_{\Ocal_K}$ with respect to the uniformiser $\pi$, which, for example, can be done via Gr\"obner bases (see \cite[Chapter 4.4]{cox2013ideals}). For each $1 \le i \le n$, let us write  $A_i = (a_{1}^{(i)} \mid \dots \mid a_{e}^{(i)})$, where $a_{1}^{(i)}, \dots, a_{e}^{(i)}$ are the columns of $A_i$. Let us also consider the matrices $B_i$ define by
	 \[
	 B_i = (a_{i}^{(1)} \mid \dots \mid a_{i}^{(n)} ) \in \mathrm{Mat}_{d \times n}(K).
	 \]
	 The matrices $B_1, \dots, B_e$ will be our bridge to Mustafin varieties. Their importance stems from the fact that, for a column vector $x = (x_1, \dots, x_n)^\top \in K^n$, we have
	 \[
	 x_1 A_1 + \dots + x_n A_n =  (B_1 x \mid \dots \mid B_e x).
	 \]
	 So the linear space $\Lcal$ is given by
	 \begin{align}
	      \Lcal &= \left \{a_1  A_1 + \dots + a_n A_n :  a_1 ,\dots, a_n \in K \right\}\\
	            &= \{ (B_1 a, \dots, B_e a) : a \in K^n \textrm{ column vector}\}. 
	 \end{align}
	 Therefore we can write $\Kcal$ as follows
	 \begin{align}
	 \Kcal &= \left(K A_1 \oplus \dots \oplus K A_d\right)\cap\Mat_{d\times e}(\mathcal{O}_K)\\ 
	       &= \left\{ (B_1 a \mid \dots \mid B_e a) : a \in K^n \right\}\cap\Mat_{d\times e}(\mathcal{O}_K).    
	 \end{align}
	 \noindent
	 In this section we study the dimension of the ${\overline{\K}}$-linear space of matrices $\overline{\Kcal} = \Kcal \mod \pi$. To this end, we will make a connection to convex hulls in Bruhat-Tits buildings to provide certificates for $\dim(\overline{\Kcal})$ to be maximal. For the rest of this section, we make the following general assumption.
	 
	 \begin{assumption}
	 	We assume that $n=d$ and for each $i$ we have $B_i \in \GL(d,K)$. Furthermore, for each $1 \leq i \leq n $, we denote $\Lambda_i = B_i^{-1}\mathcal{O}_K^d$.
	 \end{assumption}

	 \noindent
	 An important notion for our considerations will be that of the \textit{multi-projective closure of a linear space of matrices}.
	 
	 \begin{definition}
	 	For $\mathcal{L}$ as above, we define the multi-projective closure $\mathcal{MP}(\mathcal{L})$ of $\Lcal$ as the Zariski closure
	 	\begin{equation}
	 		\overline{\left\{(\alpha_1,\dots,\alpha_e)\in\left(\mathbb{P}_K^{d-1}\right)^e : (\alpha_1\mid\dots\mid\alpha_e) \in \mathcal{L}\,\textrm{and for all}\,i\,\textrm{we have}\,\alpha_i\neq \underline{0}\right\}}\subseteq\left(\mathbb{P}_K^{d-1}\right)^e.
	 	\end{equation}
	 	Analogously, we define the multi-projective closure $\mathcal{MP}(\overline{\Kcal})\subseteq\left(\mathbb{P}_{\overline{\K}}^{d-1}\right)^e$.
	 \end{definition}
	 
	 \noindent
	 Notice that the multi-projective closure of an rank metric code can be empty.

	 \begin{lemma}
	 \label{lem:projcl}
	 	If $\mathrm{dim}(\mathcal{MP}(\overline{\mathcal{K}}))=d-1$, then $\mathrm{dim}(\overline{\mathcal{K}})=\mathrm{dim}(\Kcal)$. Moreover, $\mathcal{MP}(\overline{\mathcal{K}})$ is irreducible.
	 \end{lemma}
	 
	 \begin{proof}
    Let us choose a basis $\tilde{A}_1,\dots,\tilde{A}_m$ of $\overline{\Kcal}$ and, for each $1 \leq i \leq m$, write 
    \begin{equation}
        \tilde{A}_i=(\tilde{a}_{1}^{(i)}\mid\dots\mid \tilde{a}_{e}^{(i)}).
    \end{equation}
    We see immediately that $m\le d=n$ and define the matrices $\tilde{B}_1, \dots, \tilde{B}_e$ as follows:
    \begin{equation}
        	 \tilde{B}_i = (\tilde{a}_{i}^{(1)} \mid \dots \mid \tilde{a}_{i}^{(m)} ) \in \mathrm{Mat}_{d \times m}(\overline{K}).
    \end{equation}
    This construction naturally gives rise to the rational map
    \begin{equation}
    \begin{tikzcd}
        \tilde{b}\colon\mathbb{P}_{\overline{K}}^{m-1} \arrow[dashed,"{(\tilde{B}_1,\dots,\tilde{B}_e)}"]{rrr} & & & \left(\mathbb{P}_{\overline{K}}^{d-1}\right)^e.
    \end{tikzcd}
    \end{equation}
            It follows that $\mathcal{MP}(\overline{\Kcal})=\overline{\mathrm{Im}(\tilde{b})}$ and so  that $\mathcal{MP}(\overline{\Kcal})$ is irreducible and  $\mathrm{dim}(\mathcal{MP}(\overline{\Kcal}))\le m-1$. Thus, if we have $\mathrm{dim}(\mathcal{MP}(\overline{\Kcal}))=d-1$ and since $m\le d$, we obtain $m=d$ and therefore $\mathrm{dim}(\overline{\Kcal})=m=d=\mathrm{dim}(\Kcal)$.
	\end{proof}
	 
    \begin{example}
    \label{ex:projcl}
        \begin{enumerate}
            \item If $\mathcal{L}=\langle A_1,A_2\rangle$, where $ A_1=\begin{pmatrix}
                1 & 0\\
                0 & 1
            \end{pmatrix}$
            and $A_2=\begin{pmatrix}
                0 & 1\\
                1 & 0
            \end{pmatrix}$, then we have that $\mathcal{MP(\overline{\mathcal{K}})}\subset\mathbb{P}^1\times\mathbb{P}^1$ is isomorphic to the subvariety cut out by the ideal $I=(x_{11}x_{12}-x_{21}x_{22})$, where $([x_{11}\colon x_{21}],[x_{12}\colon x_{22}])$ are the coordinates of $\mathbb{P}^1\times\mathbb{P}^1$. In other words, it is isomorphic to the diagonal subvariety of $\mathbb{P}^1\times\mathbb{P}^1$ and therefore irreducible, of dimension $1$. \cref{lem:projcl} yields then that $\mathrm{dim}(\mathcal{\overline{K}})=2$, which can also be seen immediately from the generators of $\mathcal{L}$.
            \item Let $\mathcal{L}=\langle A_1,A_2\rangle$, where
            $A_1=\begin{pmatrix}
                1 & 0\\
                0 & \pi
            \end{pmatrix}$ and $A_2=\begin{pmatrix}
                0 & 1\\
                \pi & 0
            \end{pmatrix}$. Since in this case $A_1$ and $A_2$ are also a basis for $\mathcal{K}$, we see that
            \begin{equation}
                \mathcal{\overline{K}}=\left\langle\begin{pmatrix}
                    1 & 0\\
                    0 & 0
                \end{pmatrix},\begin{pmatrix}
                    0 & 1\\
                    0 & 0
                \end{pmatrix}\right\rangle.
            \end{equation}
            In this case, we observe that $\mathcal{MP(\overline{K})}\subset\mathbb{P}^1\times\mathbb{P}^1$ is simply the point $([1\colon 0],[1\colon 0])$ in the same coordiantes as the first example.  Therefore, we obtain $\mathrm{dim}(\mathcal{MP(\overline{K})})=0$, although we see that $\mathrm{dim}(\overline{\Kcal})=2$. We thus see that the condition in \cref{lem:projcl} is only sufficient and not necessary.
        \end{enumerate}
    \end{example}

	 \begin{proposition}
	 \label{prop-metricmust}
	 	We assume that $\overline{K}$ is infinite. Let $\Gamma=\{[\Lambda_1],\dots,[\Lambda_e]\}\subseteq\mathfrak{B}_d^0$ be the finite set of homothety classes induced by $B_1,\dots,B_e$ as above and consider the embedding $\mathcal{M}(\Gamma)\subseteq\mathbb{P}(\mathcal{O}_K^d)^e$ induced by $\mathbb{P}(V)\xrightarrow{(B_1,\dots,B_e)}\mathbb{P}(\mathcal{O}_K^d)^e$. The following holds:
	 	
	 	\begin{enumerate}[label=$(\arabic*)$]
	 		\item We have $\mathcal{MP}(\overline{\Kcal})\subseteq\mathcal{M}(\Gamma)_{\overline{\K}}$.
	 		\item $\mathcal{MP}(\overline{\Kcal})$ is an irreducible component of $\mathcal{M}(\Gamma)_{\overline{\K}}$ if and only if $\mathrm{dim}(\mathcal{MP}(\overline{\Kcal}))=d-1$.
	 	\end{enumerate}
	 \end{proposition}
	 
	 \begin{proof}
	 	The proof is analogous to the proof of \cite[lemma 3.7]{hahn2020mustafin}.
	 \end{proof}
	 
	 \noindent
	 By Theorem \ref{thm-hlmust}, there exists a unique lattice class $[\Lambda]\in\mathrm{conv}(\Gamma)$, such that $ \mathcal{MP}(\overline{\Kcal})=(B_1,\dots,B_e)\left( X_{[\Lambda],\Gamma}\right)$. It is an interesting question to identify this homothety class. The following result gives a necessary criterion for the case when $[\Lambda]=[\mathcal{O}_K^d]$.

	 \begin{theorem}
	 	\label{prop-must}
	 	We assume that $\overline{K}$ is infinite. Let $\Gamma=\{[\Lambda_1],\dots,[\Lambda_e]\}\subseteq\mathfrak{B}_d^0$ be the finite set of homothety classes induced by $B_1,\dots,B_e$ as above. If $\mathcal{K}=\langle A_1,\dots,A_d\rangle_{\mathcal{O}_K}$ and $\mathrm{dim}(\mathcal{MP}(\overline{\Kcal}))=d-1$, then we have $[\mathcal{O}_K^d]\in\mathrm{conv}(\Gamma)$.
	 \end{theorem}
	 
	 \begin{proof}
	 	We first observe that if $\mathcal{K}=\langle A_1,\dots,A_d\rangle_{\mathcal{O}_K}$ and $\mathrm{dim}(\mathcal{MP}(\overline{\Kcal}))=d-1$, the $B_i$ are saturated as well, since otherwise $\mathcal{MP}(\overline{\Kcal})=\emptyset$. We then deduce that
	 	\begin{equation}
	 		\overline{\mathcal{K}}={\overline{\K}} \left(A_1\,\mathrm{mod}\,\pi\right)\oplus\dots\oplus {\overline{\K}} \left(A_d\,\mathrm{mod}\,\pi\right)
	 	\end{equation}
	 	and therefore
	 	\begin{equation}
	 		\mathcal{MP}(\overline{\Kcal})=\overline{\mathrm{Im}\left(
	 			\mathbb{P}_{\overline{\K}}^d \xdashrightarrow{(B_1\,\mathrm{mod}\,\pi,\dots,B_e\,\mathrm{mod}\,\pi)} \left(\mathbb{P}_{\overline{\K}}^d\right)^e\right)}.
	 	\end{equation}
	 	Thus, from the definition of $\mathcal{MP}(\overline{\mathcal{K}})$ and  Theorem \ref{thm-hlmust}, we obtain 
	 	\[
	 	    \mathcal{MP}(\overline{\Kcal}) =  (B_1, \ldots , B_e)_{\overline{K}} \ ( X_{[\Lambda],\Gamma} )
	 	\] 
	 	and, again from Theorem \ref{thm-hlmust}, we obtain $[\mathcal{O}_K^d]\in\mathrm{conv}(\Gamma)$.
	 \end{proof}
	 
    \begin{example}
    We note that in the setting of \cref{prop-must}, the condition $[\mathcal{O}_K^d]\in\mathrm{conv}(\Gamma)$ is only necessary but not sufficient condition. Indeed, we may consider $\mathcal{L}=\langle A_1,A_2\rangle$, where
    \begin{equation}
        A_1=\begin{pmatrix}
            1 & \pi\\
            0 & 0
        \end{pmatrix}\quad\textrm{and}\quad A_2=\begin{pmatrix}
           0 & 0\\
           \pi & 1
        \end{pmatrix}.
    \end{equation}
    We see that \begin{equation}
        \mathcal{\overline{K}}= \left\langle\begin{pmatrix}
            1 & 0\\
            0 & 0
        \end{pmatrix},\begin{pmatrix}
            0 & 0\\
            0 & 1
        \end{pmatrix} \right \rangle
    \end{equation}
    and therefore $\mathrm{dim}(\overline{\Kcal})=2$. A similar argument as in \cref{ex:projcl}(2) shows that $\mathcal{MP}(\overline{\Kcal})$ is simply a point and therefore $\dim(\mathcal{MP}(\overline{\Kcal})) = 0$. Thus, we have $\dim(\mathcal{MP}(\overline{\Kcal}) )\neq d-1=1$.\\
    We now show that for $\Gamma$ as in \cref{prop-must}, we have $[\mathcal{O}_K^2]\in{\rm conv}(\Gamma)$. From the matrices $A_1,A_2$, we obtain
    \begin{equation}
        B_1=\begin{pmatrix}
            1 & 0\\
            0 & \pi
        \end{pmatrix}\quad\textrm{and}\quad B_2=\begin{pmatrix}
            \pi & 0\\
            0 & 1
        \end{pmatrix}.
    \end{equation}
    With $(e_1,e_2)$ denoting the standard basis of $\mathcal{O}_K^2$, we thus obtain $\Gamma=\{[L_1],[L_2]\}$ as the corresponding set of homothety classes, where $L_1=\mathcal{O}_K e_1+\pi\mathcal{O}_Ke_2$ and $L_2=\pi\mathcal{O}_Ke_1+\mathcal{O_K}e_2$. We see immediately that $L_1\cap L_2=\pi\mathcal{O}_K$ and so  $[\mathcal{O}_K^2]=[\pi\mathcal{O}_K^2]\in\mathrm{conv}(\Gamma)$ as desired.
    \end{example}
    
    \begin{example}
    In this example, we illustrate that $\mathcal{K}=\langle A_1,\dots,A_n\rangle_{\mathcal{O}_K}$ is in fact necessary to deduce $[\mathcal{O}_K^d]\in\mathrm{conv}(\Gamma)$. 
    We consider $\mathcal{L}=\langle A_1,A_2\rangle$ where
    \begin{equation}
        A_1=\begin{pmatrix}
            1 & 1\\
            0 & 0
        \end{pmatrix}\quad\textrm{and}\quad
        A_2=\begin{pmatrix}
            1 & 1\\
            \pi & \pi
        \end{pmatrix}.
    \end{equation}
    We see immediately that $A_1$ and $A_2$ do not form an $\mathcal{O}_K$-basis, since 
    \begin{equation}
        \begin{pmatrix}
            0 & 0\\
            1 & 1
        \end{pmatrix}=\frac{1}{\pi}\left(A_2-A_1\right)\in\mathcal{L}.
    \end{equation}
    We also see that $\overline{\mathcal{K}}=2$ and, since $\mathcal{MP(\overline{\mathcal{K}})}$ is the diagonal in $\mathbb{P}^1\times\mathbb{P}^1$, that $\mathrm{dim}(\mathcal{MP(\overline{\mathcal{K}})})=1$. However, we obtain from $A_1$ and $A_2$ that
    \begin{equation}
        B_1=B_2=\begin{pmatrix}
            1 & 1\\
            0 & \pi
        \end{pmatrix}
    \end{equation}
    and observe that the lattice class associated to $B_1=B_2$ is different from $[\mathcal{O}_K^2]$.

    \end{example}

	 \begin{remark} 
    Let $E$ be the maximal unramified field extension of $K$ and $\mathcal{L}$ be a linear space of matrices in $\mathrm{Mat}_{d\times e}(K)$. Then $E\otimes_K\mathcal{L}$ is a linear subspace of $\mathrm{Mat}_{d\times e}(E)$ and it is easily seen that the ranks of
    \begin{equation}
        \mathcal{K}\quad\textrm{and}\quad (E\otimes_K\mathcal{L})\cap\mathrm{Mat}_{d\times e}(\mathcal{O}_E)
    \end{equation}
    are equal. By observing that $(E\otimes_K\mathcal{L})\cap\mathrm{Mat}_{d\times e}(\Ocal_E)=\mathcal{O}_E\otimes_{\mathcal{O}_K}\mathcal{K}$, we see that the dimensions of
    \begin{equation}
        \overline{\mathcal{K}}\quad\textrm{and}\quad (E\otimes_K\mathcal{L})\cap\mathrm{Mat}_{d\times e}(\Ocal_E)\mod\pi\subset\mathrm{Mat}_{d\times e}(\overline{E})
    \end{equation}
    agree. Therefore, the question of whether $\mathrm{dim}(\overline{K})=d-1$ may be detected by passing to the extension $E$ whose residue field is the separable closure of $\overline{K}$.
\end{remark}
     \noindent
	 \cref{prop-must} leads us to believe that the convex hull $\mathrm{conv}(\Gamma)$ encodes the changes of bases necessary to get an $\Ocal_K$-basis of $\Kcal$.  This is not very surprising because one of the many interesting aspects of the building $\Bfrak_d$ is that it encodes the combinatorics of base change in a linear space.
	 
	 \begin{question1}
	 	In the setting of \cref{prop-must}, can a base change of $A_1,\dots,A_d$ that yields an $\mathcal{O}_K$-basis of $\Kcal$ be detected from the convex hull $\mathrm{conv}(\Gamma)$?
	 \end{question1}
	 
\bibliographystyle{acm}
\bibliography{references}

\end{document}